\documentclass[12pt]{amsart} \textwidth=14.5cm \oddsidemargin=1cm
\usepackage{amsmath,wasysym}
\usepackage{amsxtra}
\usepackage{amscd}
\usepackage{amsthm}
\usepackage{amsfonts}
\usepackage{amssymb}
\usepackage{eucal}
\usepackage{graphics, color}
\usepackage{hyperref,mathrsfs}
\usepackage[usenames,dvipsnames]{xcolor}
\usepackage{stmaryrd}
\usepackage[all]{xy}
\usepackage{hyperref}
\usepackage{color}
\usepackage{bbm} 
\usepackage{ytableau,graphics}
\allowdisplaybreaks

\hypersetup{colorlinks,linkcolor=blue, urlcolor=blue,
citecolor=blue}

\usepackage{pgf,tikz}
\usetikzlibrary{arrows, positioning, calc, chains}
\tikzset{
	ch/.style={circle,draw,on chain,inner sep=2pt},
	chj/.style={ch,join},
	every path/.style={shorten >=4pt,shorten <=4pt}
	}
\newcommand{\dnode}[2][chj]{%
	\node[#1,label={below:#2}] (#1) {};}
\newcommand{\dnodenj}[1]{%
	\dnode[ch]{#1}}
\newcommand{\dydots}{%
	\node[chj,draw=none,inner sep=1pt] {\dots};}
\newtheorem{theorem}{Theorem}[section]
\newtheorem{lemma}[theorem]{Lemma}
\newtheorem{corollary}[theorem]{Corollary}
\newtheorem{proposition}[theorem]{Proposition}
\theoremstyle{definition}
\newtheorem{definition}[theorem]{Definition}
\newtheorem{example}[theorem]{Example}

\theoremstyle{remark}
\newtheorem{remark}[theorem]{Remark}

\numberwithin{equation}{section}

\newcommand{\qbinom}[2]{\begin{bmatrix} #1\\#2 \end{bmatrix} }

\def\A{\mathcal{A}}
\def\ff{\mathbf{f}}
\def\H{\mathbf{H}}
\def\HA{\mathbf{H}_{\mathcal{A}}}

\def\mf{\mathfrak}

\def\I{\mathbb{I}}
\def\Ib{\I_{\bullet}}
\def\Iw{\I_{\circ}}
\def\Ii{\I^\imath}
\def\Ij{\I^\jmath}
\def\la{\lambda}
\def\C{\mathbb C}
\def\N{\mathbb N}
\def\Q{\mathbb Q}
\def\Z{\mathbb Z}

\def\U{\mathbf{U}}
\def\Ui{\mathbf{U}^{\imath}}
\def\tU{\widetilde{\mathbf{U}}}
\def\tUi{\widetilde{\mathbf{U}}^{\imath}}

\def\UA{\mathbf{U}_\mathcal{A}}
\def\V{\mathbb{V}}
\def\VA{\mathbb{V}_\A}

\def \bvs{{\boldsymbol{\varsigma}}}
\newcommand{\vs}{\varsigma}

\title[Lectures on dualities ABC in representation theory]
{Lectures on dualities ABC in representation theory}

\author[Li Luo]{Li Luo}
\address{Department of mathematics,
Shanghai Key Laboratory of Pure Mathematics and Mathematical Practice,
    East China Normal University, Shanghai 200241, China}
\email{lluo@math.ecnu.edu.cn (Luo)}

\author[Weiqiang Wang]{Weiqiang Wang}
\address{Department of Mathematics\\ University of Virginia\\ Charlottesville, VA 22904}
\email{ww9c@virginia.edu (Wang)}

\begin{document}

\begin{abstract}
In these lecture notes for a summer mini-course, we provide an exposition on quantum groups and Hecke algebras, including (quasi) R-matrix, canonical basis, and $q$-Schur duality. Then we formulate their counterparts in the setting of $\imath$quantum groups arising from quantum symmetric pairs, including (quasi) K-matrix, $\imath$-canonical basis, and $\imath$Schur duality. As an application, the ($\imath$-)canonical bases are used to formulate Kazhdan-Lusztig theories and character formulas in the BGG categories for Lie (super)algebras of type A-D. Finally, geometric constructions for $q$-Schur and $\imath$Schur dualities are provided.
\end{abstract}

\maketitle
\setcounter{tocdepth}{2}
\tableofcontents

\section{Introduction}

\subsection{A quick overview}

The classic Schur duality relates the representation theories of general linear Lie algebras and the symmetric groups as well as Young tableau combinatorics of Schur functions. On one hand, it can be reformulated and then generalized to Howe dualities between classical groups. On the other hand, it admits a remarkable $q$-deformation, which allows further generalizations and interpretations.
These notes address $q$-Schur dualities of type ABC(D) as well as their categorical and geometric incarnations.

We first review some basic constructions for Hecke algebras and quantum groups of finite type, including (quasi) R-matrix, canonical bases, $q$-Schur duality and its geometric interpretation. We further use the canonical basis in a (generalized) setting of $q$-Schur duality to formulate the Kazhdan-Lusztig theory for the BGG category $\mathcal O$ of (super) type A. The materials on quantum groups and Hecke algebras are widely known  with a possible exception of the canonical basis formulation of (super) Kazhdan-Lusztig theory.

Then we present the parallel developments in the past few years on $\imath$quantum groups arising from quantum symmetric pairs of finite type, including (quasi) K-matrix, $\imath$-canonical bases, and $\imath$Schur duality. For the sake of simplicity, we have restricted ourselves to the {\em quasi-split} $\imath$quantum groups in these notes. We use the $\imath$-canonical basis in a (generalized) setting of $\imath$Schur duality to formulate the Kazhdan-Lusztig theory of (super) classical type.

The notes contain no original results. We provide or sketch proofs in the notes whenever it is practical to do so. For several major theorems, we refer to the references for complete proofs. For subjects like super Kazhdan-Lusztig theories or geometric realizations of ($\imath$-) quantum groups, our sketchy exposition hopefully helps to provide some roadmap to read the original papers.

These notes can serve as a quick introduction for graduate students to quantum groups in action; they may also help to guide researchers who are already familiar with quantum groups into the area of $\imath$quantum groups, which has been expanding rapidly in recent years.

\subsection{Quantum groups}

Since their introduction by Drinfeld and Jimbo \cite{Dr86, Jim86} independently, quantum groups have played a central role in representation theory, quantum topology, and mathematical physics. The universal R-matrix introduced by Drinfeld \cite{Dr86} provides solutions to the Yang-Baxter equation.

Just as (Iwahori-) Hecke algebras are deformations of Weyl group algebras, quantum groups are deformations of universal enveloping algebras of semisimple Lie algebras. As a profound analogue of Kazhdan-Lusztig basis for a Hecke algebra \cite{KL79} and motivated by Ringel's Hall algebra realization of half a quantum group \cite{R90}, Lusztig introduced canonical bases arising from quantum groups \cite{Lus90}.

Further generalizations and different approaches toward canonical bases for halves of quantum groups and integrable modules have been developed \cite{Ka91, Lus91}. Canonical bases for tensor products of modules and modified quantum groups were constructed in \cite{Lus92}, where a key component of R-matrix, known as quasi R-matrix, is used to define a bar involution on the tensor products.

Schur duality between a general linear Lie algebra and a symmetric group admits a $q$-deformation \cite{Jim86}, where the Lie algebra is replaced by the corresponding quantum group and the symmetric group by the Hecke algebra of type A. A different formulation of the $q$-Schur duality via $q$-Schur algebra was given in \cite{DJ89}; also cf. \cite{GL92}. As an analogue of Iwahori's realization of Hecke algebras \cite{Iw64}, a geometric realization for $q$-Schur algebras and their canonical bases via partial flag varieties of type A is given in \cite{BLM90}, and it further leads to a construction of the type A quantum group and its canonical basis from the family of $q$-Schur algebras.

Our exposition on KL basis for Hecke algebras, (quasi) R-matrix for quantum groups, and $q$-Schur duality is largely self-contained. On the other hand, we only give a very rough outline of a construction of canonical bases on $\U^-$ and finite-dimensional simple $\U$-modules. Then we present in detail the construction of a bar involution (via quasi R-matrix) and the canonical basis on a tensor product of based $\U$-modules; this in turn leads to a construction of the canonical basis on the modified quantum group $\dot\U$.

The comultiplication $\Delta$ on $\U$ here (which is the same as \cite{Ka91, Br03, CLW15, BW18a} but different from \cite{Lus93}) is chosen in these notes to make the canonical basis on the tensor power fit most naturally with the partial ordering arising from the BGG category; we have also chosen to work with $\Z[q]$-lattices instead of Lusztig's choice of $\Z[q^{-1}]$-lattices for canonical bases. These choices have led to different conventions in the notes, e.g., on quasi R-matrix, than those used in \cite{Lus93, Jan95}.

\subsection{$\imath$quantum groups}

A theory of quantum symmetric pairs $(\U, \Ui)$ was systematically developed by Gail Letzter \cite{Let99, Let02}, with Satake diagrams (which classify the real forms of simple Lie algebras) as inputs. A Satake diagram is encoded in a pair  $(\I =\Iw \cup \Ib, \tau)$ subject to some natural constraints, where $\I =\Iw \cup \Ib$ is a 2-colored partition of the vertex set of a Dynkin diagram and $\tau$ is a diagram involution (with $\tau=\text{id}$ allowed). By definition, $\Ui$ is a coideal subalgebra of $\U$ in the sense that the comultiplication $\Delta$ on $\U$ satisfies
\[
\Delta (\Ui) \subset \Ui \otimes \U,
\]
and at the $q\mapsto 1$ limit $(\U, \Ui)$ becomes the classical symmetric pairs $(\mf g, \mf g^\theta)$, where $\theta$ is an involution of a semisimple Lie algebra $\mf g$.

We shall refer to $\Ui$ as an {\em $\imath$quantum group}. An $\imath$quantum group is {\em quasi-split} if $\Ib =\emptyset$, and it is {\em split} if $\Ib =\emptyset$ and $\tau=\text{id}$. Quantum groups can be regarded as examples of quasi-split $\imath$quantum groups. Indeed, the embedding $(\omega \otimes 1) \circ \Delta : \U \rightarrow \U \otimes \U$ makes $\U$ a coideal subalgebra of $\U \otimes \U$, and hence $(\U \otimes \U, \U)$ is a quantum symmetric pair of diagonal type. This is in the same spirit as viewing  a complex Lie group as a real group.

The $\imath$quantum groups come with parameters, which correspond to a family of embeddings of $\Ui$ into $\U$.
A Serre presentation for an arbitrary $\imath$quantum group of finite type was given by Letzter \cite{Let02}; however Letzter's convention for quantum groups is not quite standard and differs from those used by \cite{Lus93, CP94, Jan95}. A Kac-Moody generalization of $\imath$quantum groups, which is consistent with Lusztig's convention on braid group symmetries, was provided by Kolb \cite{Ko14}.

An $\imath$-program was initiated in 2013 by Huanchen Bao and the second-named author \cite{BW18a}, aiming at generalizing ``{\em all}'' fundamental (algebraic, geometric, and categorical) constructions from quantum groups to $\imath$quantum groups.
A new $q$-Schur duality (called {\em $\imath$Schur duality}) between Hecke algebra of type B and the $\imath$quantum group $\Ui$ of type AIII (with suitable parameters) was constructed {\em loc. cit.}

As an $\imath$-analogue of quasi R-matrix, a notion of {\em quasi K-matrix} (initially called an {\em intertwiner}) was formulated  in \cite{BW18a} for the first time, and this further gives rise to a $\Ui$-module isomorphism $\mathcal T$ on a $\U$-module (called a {\em K-matrix} later). The $\imath$-canonical basis on a based $\U$-module (a.k.a. a $\U$-module with canonical basis) was constructed {\em loc. cit.}, where a new bar involution on a based $\U$-module is constructed via the quasi K-matrix. The notion of $\imath$-canonical basis was in turn motivated by the formulation of a super Kazhdan-Lusztig theory of classical type; see below.

The (quasi) K-matrix is a key new notion in the program of $\imath$-canonical basis announced in \cite{BW18a}, and it is expected to be valid in general. The quasi K-matrix $\Upsilon$ lies in a completion of $\U^-$, and its formulation via a bar involution $\psi_\imath$ on $\Ui$ (which is not a restriction of the bar involution $\psi$ of $\U$) is uniform for all types (see Theorem~\ref{thm:quasi2}):
\[
    \imath( \psi_\imath (u) ) \circ\Upsilon=\Upsilon\circ \psi ( \imath (u)), \quad \mbox{for all $u\in\Ui$},
 \]
where $\imath: \Ui \rightarrow \U$ denotes the inclusion.

The existence proof of $\Upsilon$ was initially carried out in \cite{BW18a} for quasi-split $\imath$quantum groups of type AIII. The extension to quasi-split $\imath$quantum groups was straightforward and known to Bao and Wang (\cite[Remark~4.9]{BW18b}). Balagovic and Kolb \cite{BK19} have established the existence of (quasi) K-matrix in great generality of $\imath$quantum groups of Kac-Moody type, which requires much technical preparations on a bar involution (cf. \cite{BK15}); they renamed ``intertwiner" to ``quasi K-matrix".

A general theory of $\imath$-canonical bases on based $\U$-modules and modified $\imath$quantum groups has been developed in \cite{BW18b, BW18c}, for which the quasi K-matrix in the generality of \cite{BK15, BK19} is used together with a key integrality property. In the setting of quantum symmetric pairs of diagonal type, we recover Lusztig's construction of canonical basis on tensor products of $\U$-modules. A further generalization is made in \cite{BWW20}, where an $\imath$-canonical basis  is constructed on the tensor product of a based $\Ui$-module and a based $\U$-module.

In these notes, we restrict ourselves to {\em quasi-split} $\imath$quantum groups of {\em finite type}, which greatly simplify the Serre presentation; on the other hand, this class of $\imath$quantum groups is more than sufficient for the applications to $\imath$Schur duality, KL theory, and a geometric realization via flag varieties.

\subsection{Applications to super Kazhdan-Lusztig theory}

Kazhdan-Lusztig  (KL) basis (and respectively, parabolic KL) of a Hecke algebra $\H_q(W)$ has been used to formulate a character formula for the principal block (and respectively, integral singular blocks) of the BGG category $\mathcal O$ of a semisimple Lie algebra with Weyl group $W$. This is the celebrated Kazhdan-Lusztig  conjecture \cite{KL79}, now a theorem proved in \cite{BB81, BK81}.

However, this formulation does not extend to Lie superalgebras such as general linear or ortho-symplectic Lie superalgebras, as the Weyl groups of these superalgebras do not control the linkage in the corresponding BGG category $\mathcal O$.

Consider $\U = \U_q(\mf{sl}_N)$ and its natural representation $\V$ (and its dual representation $\V^*$); we allow $N=\infty$. Then $\V^{\otimes m}$, or more generally $\V^{\otimes m} \otimes  \V^{* \otimes n}$, admits a standard basis (which is given by the tensor products of the natural basis of $\V$ and $\V^*$) as well as a canonical basis. A basic fact here \cite{FKK98} is that the canonical basis on $\V^{\otimes m}$ coincides with the KL basis on $\V^{\otimes m}$ when viewed as a direct sum of permutation modules over the Hecke algebra $\H_q(\mf S_m)$.

Let $\mathfrak{gl}(m|n)$ be the general linear Lie superalgebra with its standard triangular decomposition. We form the  BGG category $\mathcal O^{m|n}$ of $\mathfrak{gl}(m|n)$-modules (with weights in the weight lattice $X^{m|n}$), cf. \cite{CW12}. According to Brundan's conjecture \cite{Br03} (which is now a theorem due to \cite{CLW15}), the character formulas for the tilting and irreducible modules in $\mathcal O^{m|n}$ are described in terms of the canonical and dual canonical bases of $\V^{\otimes m} \otimes  \V^{* \otimes n}$ via the following isomorphism; see \eqref{eq:Psi}:
\begin{align*}
\Psi: [\mathcal{O}^{m|n}] \stackrel{\cong}{\longrightarrow} \V_\Z^{\otimes m} \otimes \V_{\Z}^{*\otimes n}.
\end{align*}
 In other words, the entries of the transition matrix from the canonical basis to the standard basis for $\V^{\otimes m} \otimes  \V^{* \otimes n}$ are the super KL polynomials for the formulation of a KL theory of super type A.

On the other hand, as a module over the $\imath$quantum group $\Ui_q(\mf{sl}_N)$ of type AIII (which we can take $N = \infty$), $\V^{\otimes m}$, or more generally $\V^{\otimes m} \otimes  \V^{* \otimes n}$, admits an $\imath$-canonical basis. The basis for $\V$ are parameterized by integers or half-integers, depending on the parity of $N$. Here we shall choose the parameters of $\Ui_q(\mf{sl}_N)$ corresponding to $p=q$ (or $p=1$) in the relevant $\imath$Schur duality.

Let $\mathfrak{osp}(2m+1|2n)$ be the ortho-symplectic Lie superalgebra (of type B) with its standard triangular decomposition. Then we can form the BGG category $\mathcal O^{m|n}_{\mf b}$ (and respectively, $\mathcal O^{m|n}_{\mf b, \frac12}$) of $\mathfrak{osp}(2m+1|2n)$-modules with weights in the ``integer weight" lattice \eqref{XZ} (and respectively, ``half-integer" weight lattice \eqref{XZhalf}), cf. \cite{CW12}. A main theorem of \cite{BW18a} states that the character formulas for the tilting and irreducible modules in $\mathcal O^{m|n}_{\mf b}$ and $\mathcal O^{m|n}_{\mf b, \frac12}$ are described in terms of the canonical and dual $\imath$-canonical bases in  $\V^{\otimes m} \otimes  \V^{* \otimes n}$ (with the choice of parameters $p=q$). In other words, the entries of the transition matrix from the $\imath$-canonical basis to the standard basis for $\V^{\otimes m} \otimes  \V^{* \otimes n}$ are the super KL polynomials for the formulation of a KL theory of super type $B$.

Bao's theorem \cite{Bao17} asserts that a KL theory of super type D for Lie superalgebra $\mathfrak{osp}(2m|2n)$ is formulated in terms of the canonical and dual $\imath$-canonical bases in  $\V^{\otimes m} \otimes  \V^{* \otimes n}$ (with the choice of parameter $p=1$).

We refer to the original papers \cite{CLW15, BW18a, Bao17} for the proofs regarding these super KL theories. One main ingredient used is the {\em super duality} developed by Cheng, Lam and the second author (cf. \cite[Chapter 6]{CW12}). The super duality is an equivalence between parabolic BGG categories for Lie algebras and Lie superalgebras of infinite rank.

\subsection{Further developments on $\imath$quantum groups}
\label{subsec:out}

There has been an explosive activity on $\imath$quantum groups in the last few years, including many works which go beyond the scope of these notes. It would require another extensive survey article or a book to cover some of these additional topics. Below we sketch some of these developments, so an interested reader can look up the references to explore further.
(Various aspects of Drinfeld-Jimbo quantum groups have been well known for decades, and hence we refrain from repeating much here.)

The $\imath$-analogue of Yang-Baxter equation is known as the reflection equation. Just as R-matrices arising from quantum groups provide solutions to Yang-Baxter equation \cite{Dr86}, K-matrices arising from $\imath$quantum groups provide solutions to reflection equation \cite{BK19}. The K-matrix has been further used \cite{Ko20} to construct a braided module category of an $\imath$quantum group $\Ui$ of finite type over the braided monoidal category of $\U$. The K-matrix can also be applied to formulate a Kohno-Drinfeld type theorem for $\imath$quantum groups \cite{DNTY20}.
Similar to the quasi R-matrix, the quasi K-matrix admits (modulo some conjectural steps) a factorization into a product of rank one factors; cf. \cite{DK19}.


A geometric realization of the modified $\imath$quantum groups and $\imath$-canonical bases (along the line of \cite{BLM90}; also see \cite{GV93, Lus99} in affine type A) was obtained in \cite{BKLW18, LiW18} for quasi-split type AIII and then extended to quasi-split affine type AIII in \cite{FLLLW}. Another Hecke-algebraic approach toward the modified $\imath$quantum groups and $\imath$-canonical bases was given in \cite{FLLLWb}, generalizing the affine type A construction in \cite{DF15}.

As an $\imath$-analogue of Nakajima quiver varieties, a class of ``$\imath$quiver varieties" (corresponding to quasi-split $\imath$quantum groups of finite type) has been constructed in \cite{Li19}; it includes the cotangent bundles of partial flags of type B/C as a basic example (compare \cite{BKLW18}).

The $\imath$Schur duality has been extended to the affine type \cite{FLW20} (also cf. \cite{FLLLW}). A further extension and formulation via affine $q$-Schur algebras of arbitrary type can be found in \cite{CLW20b}.

A  (global) crystal basis theory of $\imath$quantum groups $\Ui$ of type AIII (under some restrictive assumption on unequal parameters) has been developed in \cite{Wat18}; for recent advance in type AI see \cite{Wat21}. Letzter \cite{Let19} has constructed a Cartan subalgebra in any $\imath$quantum group of finite type (which gives the right classical limit), and this plays a basic role in a further development of a highest weight theory for a class of $\imath$quantum groups; see \cite{Wat19}.

There has been an $\imath$Hall algebra realization of the quasi-split $\imath$quantum groups in a program developed by Lu and the second author, cf. \cite{LuW19a, LuW20}. The universal $\imath$quantum groups were introduced therein with additional central elements, and $\imath$quantum groups with various parameters \`a la Letzter are recovered by central reductions.

The $\imath$Hall algebra approach  has provided a reflection functor realization of braid group symmetries on (universal) $\imath$quantum groups \cite{LuW19b}; earlier constructions of braid group symmetries of $\imath$quantum groups of finite type required computer calculations \cite{KP11} (see however \cite{Dob19}).

For the modified $\imath$quantum groups of $\mathfrak{sl}_2$, the $\imath$-canonical basis (known as $\imath$divided powers) admits explicit formulas \cite{BW18a, BeW18}. Variants of the $\imath$divided powers play a fundamental role in the Serre presentations, Serre-Lusztig relations, braid group symmetries, and $\imath$Hall algebra realization for $\imath$quantum groups \cite{CLW18, CLW20, LuW20}.

A first step toward categorification of modified $\imath$quantum groups was taken in \cite{BSWW}; this generalizes the Khovanov-Lauda categorification of modified quantum groups of type A.


Yet a new star product deformation approach toward realizing (quasi-split) $\imath$quantum groups has been developed in \cite{KY20}.




%
%
\subsection{Organization}

Sections~\ref{sec:Schur}--\ref{sec:qSchur} contain more standard materials on quantum groups, Hecke algebras, and their canonical bases; we put a special emphasis on type A and $q$-Schur duality.
Sections~\ref{sec:iQG}--\ref{sec:geom} cover the $\imath$quantum groups (in particular, of type AIII), $\imath$Schur duality, and applications to KL theory and character formulas in super category $\mathcal O$ of type ABCD.

The (classical) Schur duality between the general linear Lie algebra and the symmetric group is described in Section~\ref{sec:Schur}.
In Section~\ref{sec:QG}, we formulate the quasi R-matrix as an intertwiner between the comultiplication $\Delta$ and its bar conjugate $\overline{\Delta}$. Then we upgrade it to the R-matrix and show that it induces a $\U$-module isomophism on tensor products and provides a solution to Yang-Baxter equation.

In Section~\ref{sec:Hecke}, we construct the Kazhdan-Lusztig basis for a Hecke algebra; detailed proofs are provided.

In Section~\ref{sec:CB}, we formulate the canonical bases for halves of quantum groups, finite-dimensional simple $\U$-modules, their tensor products, and the modified quantum groups. The quasi R-matrix is used in defining the bar involution on a tensor product. Some main steps of proofs are outlined.

In Section~\ref{sec:qSchur}, we formulate the $q$-Schur duality via explicit actions of quantum group and Hecke algebra of type A, and  explain that the Hecke algebra action can be realized by R-matrix. We show the bar involution on the tensor power $\V^{\otimes m}$ of the natural representation of $\U$ can be interpreted from either quantum group or Hecke algebra side. Then we show the canonical  basis on $\V^{\otimes m}$ as a tensor product of $\U$-modules coincides with the type A KL basis from the Hecke algebra side.

In Section~\ref{sec:iQG}, we review the basic constructions of quasi-split $\imath$quantum groups and formulate the quasi K-matrix as an $\imath$-generalization of the quasi R-matrix. The quasi K-matrix is used to define a new bar involution $\psi_\imath$ and then $\imath$-canonical basis on a based $\U$-module. This generalizes Lusztig's construction of canonical bases on tensor products. We sketch the construction of $\imath$-canonical basis on a modified $\imath$quantum group.

In Section~\ref{sec:iSchur}, we describe in detail the $\imath$quantum groups $\Ui$ of type AIII. We establish an $\imath$Schur duality between $\Ui$ and Hecke algebra of type B in 2 parameters $q, p$, with the most important cases being $p=q$ or $p=1$. Then we show that the $\imath$-canonical basis on the $\Ui$-module $\V^{\otimes m}$ for $p=q$ (respectively, $p=1$) coincide with KL basis of type B (and respectively, type D).

In Section~\ref{sec:KL}, we first review the formulation of KL theory and character formulas in a BGG category via the KL basis of Hecke algebras. The $q$-Schur (and respectively, $\imath$Schur) dualities allows a reformulation of KL theory in type A (and respectively, type B or D) in terms of canonical basis (respectively, $\imath$-canonical basis with $p=q$ or $p=1$) on $\V^{\otimes m}$. We then describe a formulation of KL theory of super type A-D in terms of ($\imath$-) canonical bases.

In  Section \ref{sec:geom}, we realize the Hecke algebra as a convolution algebra of (complete) flag varieties, the $q$-Schur algebra (respectively, the $\imath$Schur algebra) as a convolution algebra of partial flag varieties of type A (and respectively, type B). This leads to a realization of the modified ($\imath$-) quantum groups.

\vspace{2mm}

\noindent{\bf Acknowledgment.}
This is an expanded version of lecture notes for WW's zoom mini-course in summer 2020 organized by LL at ECNU. We thank ECNU for its support. WW thanks his collaborators and friends on the $\imath$-program, especially Huanchen Bao, for insightful collaborations, on which these lecture notes are partially built on. We thank Hideya Watanabe for his careful reading and corrections.

\section{Schur duality}
  \label{sec:Schur}

In this section, we will review the classical Schur duality, which establishes a double centralizer property between the general linear Lie algebra $\mathfrak{gl}(V)$ and the symmetric group $\mathfrak{S}_m$. The Schur duality provides a decomposition of the tensor space $V^{\otimes m}$ as a $(\mathfrak{gl}(V),\mathfrak{S}_m)$-bimodule.
The materials in this section are standard; cf. \cite{H92, Mac95, CW12}.

\subsection{General linear Lie algebra and symmetric group}

Let $\mathfrak{g}=\mathfrak{gl}_N$ be the general linear Lie algebra over $\mathbb{C}$. Let $\mf g= \mf n^- \oplus \mf h \oplus \mf n$ be the triangular decomposition into lower triangular, diagonal, and upper triangular matrices. Its root system is
$$\Phi=\{\epsilon_i-\epsilon_j~|~1\leq i\neq j\leq N\}$$
with positive roots $\{\epsilon_i-\epsilon_j~|~1\leq i<j\leq N\}$, and its simple system is
$$\Pi=\{\epsilon_i-\epsilon_{i+1}~|~1\leq i\leq N-1\}.$$
Each finite dimensional simple $\mathfrak{g}$-module $L(\lambda)$ is a highest weight module of highest weight $\lambda\in X^+$, where $$X^+=\{\lambda=\sum_{i=1}^N\lambda_i\epsilon_i\in\mathfrak{h}^*~|~\lambda_i-\lambda_{i+1}\in\N,i=1,\ldots,N-1\}$$ is the set of dominant integral weights.
Let $V=L(\epsilon_1)$ be the natural representation of $\mathfrak{g}$ with standard basis $\{v_1,v_2,\ldots,v_N\}$.

Let $\mathfrak{S}_m$ be the symmetric group of $m$ letters. The $m$-fold tensor space $V^{\otimes m}$ admits a left $\mathfrak{g}$-module structure via
$$x(v_{i_1}\otimes v_{i_2}\otimes\cdots\otimes v_{i_m})=
\sum_{a=1}^m v_{i_1}\otimes \ldots \otimes v_{i_{a-1}} \otimes x v_{i_a}\otimes  v_{i_{a+1}} \otimes \cdots\otimes v_{i_m}
$$
for any $x\in\mathfrak{g}$, and a right $\mathfrak{S}_m$-module structure via
$$
(v_{i_1}\otimes v_{i_2}\otimes\cdots\otimes v_{i_m})\sigma =v_{i_{\sigma(1)}}\otimes v_{i_{\sigma(2)}}\otimes\cdots\otimes v_{i_{\sigma(m)}} \quad \mbox{for  $\sigma\in\mathfrak{S}_m$.}
$$
In other words, the tensor space $V^{\otimes m}$ is a $(U(\mathfrak{g}),\mathbb{C}\mathfrak{S}_m)$-bimodule, for which we denote
\begin{align}  \label{eq:bimodule}
U(\mathfrak{g})\stackrel{\Phi}{\curvearrowright} V^{\otimes m}\stackrel{\Psi}{\curvearrowleft} \mathbb{C}\mathfrak{S}_m.\end{align}

\subsection{Double centralizer property}

\begin{lemma}\label{lem:1}
  Let $W$ be a finite dimensional vector space over $\mathbb{C}$. Assume $B$ is a semisimple subalgebra of $\mathrm{End}(W)$. Let $A=\mathrm{End}_B(W)$. Then
  \begin{itemize}
    \item[(1)] $A$ is semisimple;
    \item[(2)] $\mathrm{End}_A(W)=B$;
    \item[(3)] As an $A\otimes B$-module, $W$ is multiplicity-free.
  \end{itemize}
\end{lemma}
\begin{proof}
  Let $\{V_a\}$ be all the pairwise non-isomorphic simple $B$-modules. Then $B=\bigoplus_{a}\mathrm{End}(V_a)$ since $B$ is semisimple. So as a $B$-module,
  \begin{equation}\label{eq:1}
    W=\bigoplus_a U_a\otimes V_a, \quad\mbox{where $U_a=\mathrm{Hom}_B(V_a,W)$}.
  \end{equation}

 (1). By Schur's Lemma, $\mathrm{Hom}_B(V_a, V_b) \cong \delta_{a,b} \C \mathrm{id}_{V_a}$, and we have
    \begin{align*}
      A=\mathrm{End}_B(W)&=\mathrm{End}_B \Big(\bigoplus_a(U_a\otimes V_a) \Big)\\
      &\cong \bigoplus_a\mathrm{End}(U_a)\otimes \mathrm{id}_{V_a}\cong\bigoplus_a\mathrm{End}(U_a).
    \end{align*}
    This shows $A$ is semisimple. Moreover, $\{U_a\}$ are all the pairwise non-isomorphic simple $A$-modules.

(2). Similar to the above calculation, we have
    \begin{align*}
      \mathrm{End}_A(W)&=\mathrm{End}_A \Big(\bigoplus_a(U_a\otimes V_a) \Big)\\
      &\cong \bigoplus_a\mathrm{id}_{U_a}\otimes\mathrm{End}(V_a)\cong\bigoplus_a\mathrm{End}(V_a)=B.
    \end{align*}

The formula \eqref{eq:1} can now be viewed as a decomposition of the $(A\otimes B)$-module $W$. This proves (3).
\end{proof}

\begin{theorem}[Schur duality I]\label{schurduality1}
  The images of $\Phi$ and $\Psi$ in \eqref{eq:bimodule} form double centralizers, i.e.,
  \begin{align*}
    \Phi(U(\mathfrak{g}))= &\mathrm{End}_{\mathfrak{S}_m}(V^{\otimes m}),\\
    &\mathrm{End}_{\mathfrak{g}}(V^{\otimes m})^{\mathrm{op}} =\Psi(\mathbb{C}\mathfrak{S}_m).
  \end{align*}
\end{theorem}
\begin{proof}
  It is easy to check the actions of $(\mathfrak{g},\Phi)$ and $(\mathfrak{S}_m,\Psi)$ on $V^{\otimes m}$ commute. We skip the proof of double centralizer property, and refer to \cite[Theorem~ 3.10]{CW12} for details.
\end{proof}

\subsection{Partitions and polynomial weights}

Denote by $\mathcal{P}_m$ the set of partitions of $m$. For $\lambda=(\lambda_1,\ldots,\lambda_l)\in\mathcal{P}_m$ with $\lambda_l>0$, we call $l$ the length of $\la$ and denote $\ell(\lambda)=l$. We identify $(\lambda_1,\ldots,\lambda_l)=(\lambda_1,\ldots,\lambda_l,0,\ldots,0)$ as partitions.
Let
$$
\mathcal{P}_m(N)=\Big\{\lambda=(\lambda_1,\ldots,\lambda_N)~|~\lambda_1\geq\cdots\geq\lambda_N\geq0, \sum_{i=1}^N\lambda_i=m \Big\}
$$
be the set of partitions of $m$ of length not exceeding $N$.

Partitions $\lambda$ can be visualized as Young diagrams; we shall identify them. For example,
$$(5)=\ydiagram{5}  \qquad
(1^3)=\ydiagram{1,1,1} \qquad
  (2,2,1)=\ydiagram{2,2,1} \quad .
  $$

For symmetric group $\mathfrak{S}_m$, there is a one-to-one correspondence:
\begin{align*}
  \{\mbox{conjugacy classes of } \mf S_m\} \quad &\longleftrightarrow \quad \mathcal{P}_m,\\
  \mbox{conjugacy classes of cycle type $\lambda$}\quad  &\longmapsfrom \quad \lambda.
\end{align*}
The irreducible $\mathfrak{S}_m$-modules over $\mathbb{C}$ are parametrized by $\mathcal{P}_m$ and known as
Specht modules; we shall denote them by $S^\lambda$ for $\lambda\in\mathcal{P}_m$. For example,
\begin{align*}
S^{(m)} &=\mbox{ trivial module}, \qquad
S^{(1^m)} =\mbox{ sign module}, \\
S^{(m-1,1)} &=\mbox{reflection representation of dimension $m-1$}.
\end{align*}

On the other hand, a partition $\lambda=(\lambda_1,\lambda_2,\ldots,\lambda_N)\in\mathcal{P}_m(N)$ can be regarded as the polynomial dominant weight $\lambda=\lambda_1\epsilon_1 +\lambda_2\epsilon_2 +\cdots+\lambda_N\epsilon_N$ of $\mathfrak{gl}_N$.

There are two partial orders on partitions. One is the dominance order $\geq$ of weights on $\mathfrak{gl}_N$ side. 
For $\la, \mu \in \mathcal{P}_m(N)$,
\[
\lambda\geq\mu \Leftrightarrow \lambda-\mu\in\N\Pi=\N\Phi^+.
\]
The other is the dominance order $\unrhd$ of partitions on the symmetric group side, for $\la, \mu \in \mathcal{P}_m$, let $\lambda\rightarrow\mu$ denote that $\mu$ is obtained from Young diagram $\lambda$ by moving down a box. Then we say $\lambda\unrhd\mu$ if there exists a sequence of partitions $\la^i$, for $0\le i \le k$, such that
\[
\lambda =\la^0 \rightarrow \la^1 \rightarrow \la^2  \rightarrow \ldots \rightarrow \la^k =\mu.
\]

These two partial orders are compatible with each other.

\subsection{Weight subspaces of $V^{\otimes m}$}

Let
$$\mathcal{CP}_m(N):= \big\{\mu=(\mu_1,\mu_2,\ldots,\mu_N)~|~{\textstyle \sum_{i=1}^N} \mu_i=m, \mu_i\in\N \big\}$$
be the set of compositions of $m$ into $N$ parts, where a zero part is allowed in a composition. For $\mu \in \mathcal{CP}_m(N)$, we denote $\mu=(\mu_1,\mu_2,\ldots,\mu_N) \models m$.

Let $\mu \in \mathcal{CP}_m(N)$. The $\mu$-weight subspace $(V^{\otimes m})_\mu$ of $V^{\otimes m}$, with respect to the diagonal (Cartan) subalgebra of $\mathfrak{gl}_N$, is the subspace spanned by $\{v_{i_1}\otimes\cdots\otimes v_{i_m}\}$, where $(i_1,i_2,\ldots,i_m)$ satisfies the multiset identity
\begin{equation}\label{eq:3}
  \{i_1,i_2,\ldots,i_m \}=\{\underbrace{1,\ldots,1}_{\mu_1},\underbrace{2\ldots,2}_{\mu_2},\ldots,\underbrace{N,\ldots,N}_{\mu_N}\},
\end{equation}
i.e., $\mu_k=\sharp\{s~|~i_s=k\}$, for $1\leq k\leq N$.
The subspace $(V^{\otimes m})_\mu$ is preserved by $\mathfrak{S}_m$, as $(i_1,i_2,\ldots,i_m)$ satisfying \eqref{eq:3} forms a single $\mathfrak{S}_m$-orbit.

Denoted by $$\mathfrak{S}_{\mu}=\mathfrak{S}_{\mu_1}\times\mathfrak{S}_{\mu_2}\times\cdots\times\mathfrak{S}_{\mu_N}$$ a Young subgroup of $\mathfrak{S}_m$.
Write $$v_\mu:=v_1^{\otimes \mu_1}\otimes v_2^{\otimes \mu_2}\otimes\cdots\otimes v_N^{\otimes \mu_N}\in (V^{\otimes m})_{\mu}.$$
Then $\mathbf{1}_\mu=\mathbb{C}v_\mu$ is a trivial $\mathfrak{S}_\mu$-submodule of $(V^{\otimes m})_{\mu}$. Thanks to Frobenius reciprocity, this induces a surjective $\mathfrak{S}_m$-homomorphism $\mathrm{Ind}_{\mathfrak{S}_\mu}^{\mathfrak{S}_m}\mathbf{1}_\mu\rightarrow (V^{\otimes m})_\mu$, which is an isomorphism by a dimension counting, that is,
\[
(V^{\otimes m})_\mu \cong \mathrm{Ind}_{\mathfrak{S}_\mu}^{\mathfrak{S}_m}\mathbf{1}_\mu.
\]

%
It is known that
$$\mathrm{Ind}_{\mathfrak{S}_\mu}^{\mathfrak{S}_m}\mathbf{1}_\mu=\bigoplus_{\lambda\unrhd \bar\mu}K_{\lambda\mu}S^{\lambda},$$
where $K_{\lambda\mu}$ is the Kostka number defined to be the number of semistandard tableaux (SST) of shape $\lambda$ and content $\mu$, and $\bar\mu$ denotes the partition obtained from $\mu$ by rearrangement of its parts. If $\lambda\unrhd\mu$ for $\mu \in \mathcal{P}_m(N)$, then $\la  \in \mathcal{P}_m(N)$.

\begin{example}\label{ex:1}
Let $N\ge 5$. Let
\[
\lambda =
\ydiagram{7,5,2} \quad.
\]
Then $T=\begin{ytableau}
   1&1&2&3&3&3&5\\
   2&2&3&4&4\\
 3&4 \\
\end{ytableau}$
is an example of SST of shape $\la$ (or simply, $\la$-SST) with content
$\mathrm{cont}(T)=(2,3,5,3,1)$, which corresponds to the $\mathfrak{gl}_N$-weight $2\epsilon_1+3\epsilon_2+5\epsilon_3+3\epsilon_4+\epsilon_5$.
\end{example}

\subsection{Decomposition of $V^{\otimes m}$ as a bimodule}
\begin{example}
  The symmetric group $\mathfrak{S}_2$ is abelian, and has two 1-dimensional (trivial and sign) modules. Recall $V$ is the natural representation of $\mathfrak{gl}_N$. Then
  \begin{equation}\label{eq:2}
    V^{\otimes 2}\cong S^2 V\oplus \Lambda^2 V.
  \end{equation}
  Here $S^mV$ denotes the $m$th symmetric power of $V$ and
  $\Lambda^mV$ denotes the $m$th skew-symmetric power of $V$.
  As a $\mathfrak{gl}_N$-module, $S^2 V$ (resp. $\Lambda^2 V$) is irreducible of highest weight $2\epsilon_1$ (resp. $\epsilon_1+\epsilon_2$). So \eqref{eq:2} can be viewed as a multiplicity-free decomposition of $(\mathfrak{gl}_N,\mathfrak{S}_2)$-modules.
\end{example}

We have the following generalization of the above example to $V^{\otimes m}$.

\begin{theorem} [Schur duality II]   \label{schurduality2}
As a $(\mathfrak{gl}_N,\mathfrak{S}_m)$-module, we have
$$V^{\otimes m}\cong\bigoplus_{\lambda\in\mathcal{P}_m(N)}L(\lambda)\otimes S^{\lambda}.
$$
\end{theorem}

\begin{proof}
Note that $K_{\lambda\lambda}=1$ and $K_{\lambda\mu}=0$ unless $\lambda\unrhd\mu$. Thus, as an $\mf S_m$-module, we have
$$
V^{\otimes m}=\bigoplus_{\mu\in\mathcal{CP}_m(N)}(V^{\otimes m})_\mu\cong\bigoplus_{\lambda\in\mathcal{P}_m(N)}L^{[\lambda]}\otimes S^{\lambda},
$$
where $L^{[\lambda]}:=\mathrm{Hom}_{\mathfrak{S}_m}(S^\lambda, V^{\otimes m})$. It follows by Lemma~\ref{lem:1} and Theorem~\ref{schurduality1} that $L^{[\lambda]}$ is an irreducible $\mathfrak{gl}_N$-module, since $\Psi(\mathbb{C}\mathfrak{S}_m)$ is semisimple. Furthermore, since $L_{\lambda}^{[\lambda]}=\mathbb{C}$ and $L_{\mu}^{[\lambda]}=0$ unless $\lambda\geq\mu$, we can identify the irreducible $\mathfrak{gl}_N$-module $L^{[\lambda]}$ with $L(\lambda)$.
\end{proof}

\subsection{Schur functions}

The Schur functions have intimate connections to the irreducible $\mf S_m$-modules and the irreducible $\mathfrak{gl}_N$-modules arising from the Schur duality.

The Schur function $s_\lambda$ (in infinitely many variables $x_1,x_2,\ldots$) is defined to be
$$s_\lambda=\sum_{T\in \lambda\mbox{-SST}}X^T,$$ where $$X^T=X^{\mathrm{cont}(T)}=x_1^{\mu_1}x_2^{\mu_2}\cdots$$ if $\mathrm{cont}(T)=(\mu_1,\mu_2,\ldots)$. For example, $X^T=x_1^2x_2^3x_3^5x_4^3x_5$ if $T$ is the semistandard tableaux in Example~\ref{ex:1}.

Let $\mathrm{Rep}(\mathfrak{S}_m)$ be the category of finite dimensional $\mathfrak{S}_m$-modules. We denote its Grothendieck group by $[\mathrm{Rep}(\mathfrak{S}_m)]$, which has a $\Z$-basis $[S^\lambda]$, for $\lambda\in\mathcal{P}_m$. Let $\Lambda$ be the ring of symmetric functions in infinitely many variables $x_1,x_2,\cdots$. The (Frobenius) characteristic map is given by
\begin{align*}
\mathrm{ch}: & \bigoplus_{m=0}^{\infty}[\mathrm{Rep}(\mathfrak{S}_m)]\rightarrow \Lambda,\\
 \mathrm{ch}(\chi) & =\sum_{\mu\in\mathcal{P}_m}z_\mu^{-1}\chi_\mu p_\mu, \quad \text{ for }\chi \in [\mathrm{Rep}(\mathfrak{S}_m)],
\end{align*}
where $z_\mu$ is the order of the centralizer of a permutation in $\mathfrak{S}_m$ of type $\mu$, $p_\mu=(x_1^{\mu_1}+x_2^{\mu_1}+\ldots)(x_1^{\mu_2}+x_2^{\mu_2}+\ldots)\ldots$ is the power sums, and $\chi_\mu$ is the character value of $\chi$ at a permutation of cycle type $\mu$.

Denote by $h_\lambda$ the complete homogeneous symmetric function. The following is classic (cf. \cite{Mac95}, \cite[Appendix A.1]{CW12}).

\begin{theorem}
For any partition $\la$, we have
  \begin{itemize}
\item[(1)] $\mathrm{ch}([S^\lambda])=s_{\lambda}$;
\item[(2)] $\mathrm{ch}([\mathrm{Ind}_{\mathfrak{S}_\lambda}^{\mathfrak{S}_m}\mathbf{1}_\lambda])=h_\lambda$.
\end{itemize}
\end{theorem}

For $\lambda\in\mathcal{P}_m(N)$, the character $\mathrm{ch} L(\lambda):=\sum_{\mu}\dim L(\lambda)_\mu e^\mu$ is equal to $$\mathrm{Tr}~ x_1^{E_{11}}\cdots x_N^{E_{NN}}|_{L(\lambda)}$$ by the identification $x_i=e^{\epsilon_i}$.
By the Weyl character formula, we have
\[
\mathrm{ch} L(\lambda)=\frac{\sum_{w\in W}(-1)^{\ell(w)}e^{w(\lambda+\rho)}}{\sum_{w\in W}(-1)^{\ell(w)}e^{w(\rho)}},
\]
which can be further rewritten as the ratio of 2 determinants with the denominator being the Vandermonde determinant.
Then by one of the equivalent descriptions of the Schur functions, we can identify
\[
\mathrm{ch} L(\lambda) =s_\lambda(x_1, \ldots,x_N,0,0,\ldots).
\]

\section{Quantum groups and R-matrix}
  \label{sec:QG}

We introduce the quasi R-matrix and R-matrix for quantum groups, and show that the R-matrix provides a solution for Yang-Baxter equation. Standard references for quantum groups are books \cite{Lus93, CP94, Jan95}.

\subsection{Quantum groups}

For $n\in\Z, r\in\N$, we denote the $q$-integers and $q$-binomials by
\begin{align*}
[n]:=\frac{q^n-q^{-n}}{q-q^{-1}}, &
\qquad
[r]^! =[r][r-1]\cdots [1],
\\
\qbinom{n}{r}  &=\frac{[n][n-1]\cdots[n-r+1]}{[r]^!}.
\end{align*}

Let $\mf g$ be  a complex simple Lie algebra, with root system $\Phi$ relative to a Cartan subalgebra $\mf h$ and Weyl group $W$. Fix a simple system $\I$ of $\Phi$, and we identify $\Z \I$ with the root lattice; sometimes we write the simple roots as $\alpha_i$, for $i\in \I$.
There exists a unique $W$-invariant symmetric bilinear form  $\Z \I \times \Z \I \rightarrow \Z$, $x, y \mapsto x \cdot y$, normalized such that $\min \{i \cdot i \mid i\in \I \} =2$. Denote by $X$ the weight lattice and by $X^+$ the set of dominant integral weights. We extend the bilinear pairing on $\Z \I \times \Z \I$ to $\Z \I \times X$. It is known that $i \cdot \la \in \Z$, for all $\la \in X$, $i\in \I$.

Then $C=(c_{ij})_{i,j\in \I}$ is the Cartan matrix of $\mf g$, where $c_{ij} = 2 i\cdot j/i\cdot i$. We set
\[
q_i =q^{\frac{i\cdot i}2}, \qquad \text{ for } i\in \I.
\]
We denote by $[n]_i, \qbinom{n}{r}_i$ the variants of $[n], \qbinom{n}{r}$ where $q$ is replaced by $q_i$.


\begin{definition}
  The quantum group $\U=\U_q(\mathfrak{g})$ is a $\mathbb{Q}(q)$-algebra generated by
  $E_i, F_i, K_i, K_i^{-1}\, (i\in \I)$, subject to the following relations, for $i, j \in \I$:
  \begin{align*}
    &K_iK_i^{-1}=1=K_i^{-1}K_i, \quad K_iK_j=K_jK_i, \\
    &K_iE_jK_i^{-1}=q^{i \cdot j}E_j, \quad K_iF_jK_i^{-1}=q^{-i \cdot j}F_j, \\
    &E_iF_j-F_jE_i = \delta_{i,j} \frac{K_i-K_i^{-1}}{q_i -q_i^{-1}}, \\
      \end{align*}
and, for $i \neq j$,
  \begin{align}
   \label{Serre}
  \begin{split}
 \sum_{r=0}^{1-c_{ij}} (-1)^r E_i^{(1-c_{ij} -r)} E_j E_i^{(r)} =0,\qquad
 \sum_{r=0}^{1-c_{ij}} (-1)^r F_i^{(1-c_{ij} -r)} F_j F_i^{(r)} =0.
 \end{split}
\end{align}
\end{definition}
\noindent In \eqref{Serre} above, we have used the {\em divided powers}
\[
F_i^{(r)}=\frac{F_i^r}{[r]_i^!},\quad E_i^{(r)}=\frac{E_i^r}{[r]_i^!}.
\]
The relations \eqref{Serre} are called {\em $q$-Serre relations}, which do not show up in the rank one case:
\[
\U_q(\mathfrak{sl}_2)=\left\langle E,F,K^{\pm1}~\middle|~
  \begin{array}{c}
   KE=q^2EK, KF=q^{-2}FK, \\
   {[E,F]}=EF-FE=\frac{K-K^{-1}}{q-q^{-1}}
  \end{array}\right\rangle.
\]

\begin{remark}
  \begin{itemize}
  \item[(1)] Classically, the Serre relations are $[e_i,[e_i,e_j]]=0$ in $\mathfrak{g}$,
  or $e_i^2e_j-2e_ie_je_i+e_je_i^2=0$ in $U(\mathfrak{g})$ when $c_{ij}=-1$.

  \item[(2)]
  One can embed $\U_q(\mathfrak{sl}_N)=\langle E_i,F_i,K_i^{\pm1}~|~1\leq i\leq N-1 \rangle$ into $\U_q(\mathfrak{gl}_N)=\langle E_i,F_i,D_a^{\pm1}~|~1\leq i\leq N-1, 1\leq a\leq N\rangle$ by $E_i\mapsto E_i, F_i\mapsto F_i, K_i\mapsto D_iD_{i+1}^{-1}$, where
  \[
  D_iE_iD_i^{-1}=qE_i,\;\; D_{i}E_{i-1} D_{i}^{-1}=q^{-1}E_{i-1} , \;\; D_iE_jD_i^{-1}=E_j \,\; (j\neq i,i-1),
  \]
 and similar relations hold between $D_i$ and $F_j$ with $q$ replaced by $q^{-1}$.
  \end{itemize}
\end{remark}

  The quantum group $\U=\U_q(\mathfrak{g})$ is a Hopf algebra with comultiplication $\Delta:\U\rightarrow \U\otimes\U$  given by
      \begin{equation} \label{comult}
        \Delta(K_i)=K_i\otimes K_i,\quad \Delta(E_i)=1\otimes E_i+E_i\otimes K_i^{-1},\quad
        \Delta(F_i)=F_i\otimes 1+K_i\otimes F_i.
      \end{equation}
      (There are different versions of comultiplication in the literature.)\\

      Given $\U$-modules $M$ and $M'$, the tensor product $M\otimes M'$ is again a $\U$-module by $$u(x\otimes x'):=\Delta(u)(x\otimes x')=\sum u_1 x\otimes u_2 x',$$ where $\Delta(u)=\sum u_1\otimes u_2$.

\begin{example}
Let $\U=\U_q(\mathfrak{gl}_N)$ or $\U_q(\mathfrak{sl}_N)$.
Let
\[
\V=\bigoplus_{i=1}^N\mathbb{Q}(q)v_i
\]
be the natural representation of $\U$, i.e.,
\begin{align*}
  D_iv_i=qv_i, & \qquad  F_iv_i=v_{i+1},  \qquad  E_i v_{i+1} =v_i,   \\
  D_iv_j=v_j, & \qquad F_iv_j=0, \qquad E_i v_{j+1} = 0 \quad (j\neq i),
\end{align*}
and hence
$K_iv_i=qv_i, \;K_{i}v_{i+1} =q^{-1}v_{i+1}, \; K_iv_j =v_j \; (j\neq i,i+1).$
The $m$th tensor space $\V^{\otimes m}$ is a $\U$-module by using the comultiplication $\Delta$ repeatedly.
\end{example}

\subsection{Comultiplication and twistings}
%

\begin{lemma} \label{lem:inv}
The algebra $\U$ admits the following symmetries:
\begin{itemize}
  \item[(1)]
  Chevalley involution (of $\mathbb{Q}(q)$-algebra) $\omega: E_i\leftrightarrow F_i, K_i\leftrightarrow K_i^{-1}$;
  \item[(2)]
  anti-involution (of $\mathbb{Q}(q)$-algebra) $\sigma: E_i, F_i\ \mbox{fixed}, K_i\leftrightarrow K_i^{-1}$;
  \item[(3)]
  bar involution of $\mathbb{Q}$-algebra $\psi=\bar{\phantom{x}} : E_i, F_i \ \mbox{fixed}, K_i\mapsto K_i^{-1}, q\mapsto q^{-1}$.
\end{itemize}
\end{lemma}
Note that a $\Q$-linear map which sends $q^n \mapsto q^{-n}$ for all $n$ will be called {\em anti-linear}.

One has several bialgebra/Hopf algebra structures on $\U$ by ``twisting $\Delta$''. For example, we have
\begin{align}
  \Delta^{\mathrm{op}}& =P\circ\Delta;
   \label{Dop}\\
  {}^{\sigma}\Delta& =(\sigma\otimes \sigma)\circ\Delta\circ \sigma^{-1};
  \\
   {}^{\sigma} \Delta^{\mathrm{op}} & =(\sigma\otimes \sigma)\circ\Delta^{\mathrm{op}} \circ \sigma^{-1};
  \label{Dbartau} \\
  \overline{\Delta}: & \U\rightarrow\U\otimes\U, \quad \overline{\Delta}(u) :=\overline{\Delta(\overline{u})},
   \label{Dbar}
\end{align}
where $P(u_1 \otimes u_2) =  u_2 \otimes u_1$ and $\overline{u_1 \otimes u_2} = \overline{u_1} \otimes \overline{u_2}$, for $u_1, u_2 \in \U$.
Recall from \eqref{comult} that $\Delta (E_i) = 1\otimes E_i+E_i\otimes K_i^{-1}, \Delta(F_i)=F_i\otimes 1+K_i\otimes F_i$. Then, we obtain
\begin{align*}
\overline{\Delta}={}^{\sigma}\Delta : E_i \mapsto 1\otimes E_i+E_i\otimes K_i, &
\quad
F_i \mapsto F_i\otimes 1+K_i^{-1} \otimes F_i;
\\
\Delta^{\mathrm{op}} : E_i\mapsto E_i\otimes 1 + K_i^{-1}\otimes E_i, &
\quad
 F_i \mapsto 1\otimes F_i + F_i\otimes K_i;
\\
{}^{\sigma} \Delta^{\mathrm{op}} : E_i\mapsto E_i\otimes 1 + K_i \otimes E_i, &
\quad
 F_i \mapsto 1\otimes F_i +  F_i\otimes K_i^{-1}.
\end{align*}
(All these twisted comultiplications send $K_i\mapsto K_i\otimes K_i$.)

\subsection{A bilinear form}

Let ${}'\ff$ be a free $\mathbb{Q}(q)$-algebra generated by $\theta_i$ for $i\in \I$ associated with the Cartan datum of type $(\I, \cdot)$. As a $\mathbb{Q}(q)$-vector space, there is a direct sum decomposition
$${}'\ff=\bigoplus_{\mu\in\mathbb{N}\I}{{}'\ff}_\mu,$$
where $\theta_i$ has weight $\alpha_i$ for all $i\in \I$. We denote $|x|=\mu$ if $x\in{{}'\ff}_{\mu}$.
The tensor product ${}'\ff \otimes{}'\ff$ can be regarded as a $\mathbb{Q}(q)$-algebra with multiplication
$$(x_1\otimes x_2)(x_1'\otimes x_2')=q^{|x_2| \cdot |x_1'|}x_1x_1'\otimes x_2x_2'.$$

Let $r$ be the unique algebra homomorphism $r: {}'\ff\rightarrow{}'\ff\otimes {}'\ff$ such that
$$r(\theta_i)=\theta_i\otimes 1+1\otimes \theta_i, \quad\mbox{(for all $1\leq i<n$)}.$$
\begin{lemma} [see {\cite[Proposition~1.2.3]{Lus93}}]
  \label{lem:BF}
  There is a unique symmetric bilinear form $(\cdot,\cdot): {}'\ff\times {}'\ff\rightarrow \mathbb{Q}(q)$ such that
  \begin{itemize}
    \item[(1)] $(1,1)=1$;
    \item[(2)] $(\theta_i,\theta_j)=\delta_{ij}(1-q_i^2)^{-1}$ for all $i,j \in \I$;
    \item[(3)] $(x,yy')=(r(x),y\otimes y')$ for all $x,y,y'\in {}'\ff$;
    \item[(4)] $(xx',y)=(x\otimes x',r(y))$ for all $x,x',y\in{}'\ff$.
  \end{itemize}
  Here the bilinear form $(\cdot,\cdot)$ on ${}'\ff\otimes {}'\ff$ is given by
  $(x_1\otimes x_2,x_1'\otimes x_2')=(x_1,x_1')(x_2,x_2')$.
\end{lemma}

Let $\mathbf{Rad}$ be the radical of the bilinear form $(\cdot,\cdot)$ on ${}'\ff$. The quotient algebra $\ff={{}'\ff}/\mathbf{Rad}$ still has a decomposition (as a $\mathbb{Q}(q)$-vector space) $\ff=\bigoplus_{\mu}\ff_\mu$, where $\ff_\mu$ is the image of ${}'\ff_\mu$. We denote $|x|=\mu$ if $x\in\ff_{\mu}$. The bilinear form $(\cdot,\cdot)$ on ${}'\ff$ induces a non-degenerate symmetric bilinear form on $\ff$, denoted again by $(\cdot,\cdot)$. We denote again by $\theta_i$ the image of $\theta_i$ in $\ff$.

Let $\U^-$ and $\U^+$ be the $\mathbb{Q}(q)$-subalgebra of $\U$ generated by $F_i$ and $E_i$, respectively, for $i\in \I$.
It is known (see \cite{Lus93}) that there are $\mathbb{Q}(q)$-algebra isomorphisms
$${\cdot}^{-}:\ff\cong\U^- \ \mbox{such that}\ \theta_i^-=F_i;
\quad \mbox{and}\quad {\cdot}^{+}:\ff\cong\U^+ \ \mbox{such that}\  \theta_i^+=E_i.$$
Denote $\U^-_\mu$ (resp. $\U^+_\mu$) the isomorphic copy of $\ff_\mu$ in $\U^-$ (resp. $\U^+$).

\subsection{Quasi R-matrix}

Let $\widehat{\U^+ \otimes \U^-}$ be the completion of $\U^+ \otimes \U^-$.
Recall from \eqref{Dbar} that $\overline{\Delta} \neq \Delta$.

\begin{theorem}  \label{thm:quasi}
  There is a unique family of elements $\Theta_\mu\in\U^+_\mu\otimes \U^-_\mu$, $(\mu\in\mathbb{N}\I)$, such that $\Theta_0=1\otimes 1$ and
  $\Theta=\sum_\mu\Theta_\mu\in \widehat{\U^+ \otimes \U^-}$ satisfies
  \begin{equation}\label{eq:6}
    \Delta(u)\Theta=\Theta\overline{\Delta}(u) \quad\mbox{for all $u\in\U$}.
  \end{equation}
More precisely,
  \begin{equation}\label{eq:dual}
\Theta_\mu=\prod_{i\in \I} (-q_i)^{- \mu_i}\sum_{a}u_a^+\otimes w_a^-,
\end{equation}
where $\mu =\sum_i \mu_i \alpha_i$, $\{u_a\}$ is a basis of $\ff_\mu$ and $\{w_a\}$ is its dual basis with respect to $(\cdot,\cdot)$.
\end{theorem}

\begin{proof}
See {\cite[Theorem~4.1.2]{Lus93}}.
\end{proof}
We call $\Theta$ the {\em quasi R-matrix} for $\U$.
Note that $\Theta_\mu\in\U^-_\mu\otimes \U^+_\mu$ in \cite{Lus93} (instead of being in $\U^+_\mu\otimes \U^-_\mu$ here), as Lusztig's comultiplication corresponds to $ {}^{\sigma} \Delta^{\mathrm{op}}$ in \eqref{Dbartau} (instead of $\Delta$ in \eqref{comult}). The sign convention in $q$-powers in Lemma~\ref{lem:BF}(2) (different from Lusztig) has led to a different sign in \eqref{eq:dual}.

\begin{corollary}  \label{cor:Theta}
We have
  $\overline{\Theta}=\Theta^{-1}$.
\end{corollary}

\begin{proof}
By replacing $u$ in \eqref{eq:6} by $\overline{u}$, applying the bar involution on both sides and then rewriting, we obtain that $\Delta\circ\overline{\Theta}^{-1}=\overline{\Theta}^{-1}\circ\overline{\Delta}$. Moreover, $\overline{\Theta}^{-1}_0=1\otimes 1$. By the uniqueness in Theorem~\ref{thm:quasi}, we obtain $\overline{\Theta}^{-1} =\Theta$, or equivalently, $\overline{\Theta}=\Theta^{-1}$.
\end{proof}

\subsection{R-matrix}
\label{subsec:R}

Given finite dimensional $\U$-modules $M$ and $M'$, we know $M\otimes M'\cong M'\otimes M$ as $\U$-modules because they are both completely reducible (over $\mathbb{Q}(q)$) and have the same character. But in contrast to the Lie algebra setting, $P: M\otimes M'\rightarrow M'\otimes M$ is not an isomorphism of $\U$-modules, since $\Delta$ is not co-commutative, i.e., $\Delta^{\mathrm{op}} \neq \Delta$; cf. \eqref{Dop}. Drinfeld invented the notion of R-matrix, in place of $P$, to do the job.

Recall $X$ is the weight lattice. The following is elementary.
\begin{lemma}
  \label{lem:2}
  There exists $f: X\times X \rightarrow \mathbb{Q}(q)$ such that, for any $\lambda,\mu \in X$ and $\nu \in \Z \I$,
  \[
  \left\{\begin{array}{l}
    f(\lambda+\nu,\mu)=q^{\nu \cdot \mu}f(\lambda,\mu),\\
    f(\lambda,\mu+\nu)=q^{\nu \cdot\lambda}f(\lambda,\mu).
  \end{array}
  \right.
 \]
 \end{lemma}

Let $M$ and $M'$ be finite dimensional $\U$-modules. The map $f$ induces a $\Q(q)$-linear map
  $\widetilde{f}: M\otimes M'\rightarrow M\otimes M'$ such that
  $$\widetilde{f}(x\otimes x')=f(\lambda,\mu)x\otimes x', \quad \mbox{for $x\in M_\lambda, x'\in M_\mu'$}.$$

\begin{theorem}[R-matrix]\label{thm:R}
The following identities hold:
  \begin{itemize}
    \item[(1)] $\overline{\Delta}(u)\circ \widetilde{f}=\widetilde{f}\circ \Delta^{\mathrm{op}}(u), \forall u\in \U$.
    \item[(2)] Let $\Theta^f=\Theta\circ \widetilde{f}$. Then $\Delta(u)\circ\Theta^f=\Theta^f\circ\Delta^{\mathrm{op}}(u)$, for all $u\in \U$.
    \item[(3)] $R:=\Theta\circ\widetilde{f}\circ P: M\otimes M'\rightarrow M'\otimes M$ is a $\U$-module isomorphism.
  \end{itemize}
\end{theorem}

\begin{proof}[Sketch of a proof] (See \cite[Chapter 7 and \S3.13-3.14]{Jan95} for details, up to some different sign convention on $q$-powers.)

(1) By a simple reduction it suffices to verify the identity on generators $u=E_i,F_i,K_i$. This is achieved by a direct computation, and the precise formulas for $f$ in Lemma~\ref{lem:2} actually arise from this computation.

(2) It follows from Part (1) and \eqref{eq:6} formally that $$\Delta(u)\circ\Theta^f=\Delta(u)\Theta\circ\widetilde{f}=\Theta\overline{\Delta}(u)\circ\widetilde{f}=
  \Theta\widetilde{f}\Delta^{\mathrm{op}}(u)=\Theta^f\circ\Delta^{\mathrm{op}}(u).$$

(3) Clearly, $R$ is an isomorphism of vector spaces. Let $v=m\otimes m' \in M\otimes M'$. Then, $P(uv)=P(\Delta(u)v)=\Delta^{\mathrm{op}}(u)P(v)$, for $u\in \U$. Therefore, by Part (2),
 \[
 R(uv)=\Theta^f\circ P(uv)=\Theta^f\circ\Delta^{\mathrm{op}}(u)P(v)=\Delta(u)\Theta^fP(v)=\Delta(u)R(v),
 \]
that is, $R$ is a $\U$-module homomorphism.
\end{proof}

\subsection{Yang-Baxter equation}

Let $M, M', M''$ be finite-dimensional $\U$-modules. Recall from Theorem~\ref{thm:R} that $\Theta^f \in\widehat{\U\otimes\U}$, a certain completion of $\U\otimes\U$. Denote by $R_{ij}$ $(1\leq i<j\leq3)$, a version of the R-matrix $R$ supported at the $i$th and $j$th factors, which gives rise to a $\U$-module isomorphism from $M\otimes M'\otimes M''$,
e.g.,
$R_{12}=R\otimes 1:M\otimes M'\otimes M''\rightarrow M'\otimes M\otimes M''$.

\begin{theorem} [Yang-Baxter Equation] \label{thm:YB}
We have the identification of $\U$-module isomorphisms:
\[
R_{12}\circ R_{23}\circ R_{12}=R_{23}\circ R_{12}\circ R_{23}: M\otimes M'\otimes M''\rightarrow M''\otimes M'\otimes M,
\]
 or equivalently
\[
\Theta^f_{12}\circ\Theta^f_{13}\circ\Theta^f_{23}=\Theta^f_{23}\circ\Theta^f_{13}\circ\Theta^f_{12}:
  M\otimes M'\otimes M''\rightarrow M\otimes M'\otimes M''.
\]
\end{theorem}

\begin{proof}
  See \cite[\S7.4-7.5, \S3.17]{Jan95}.
\end{proof}


Let $V$ be a finite-dimensional $\U$-module. Let
\begin{align}
  \label{eq:Ri}
R_i=1^{\otimes i-1}\otimes R_{i,i+1}\otimes 1^{\otimes m-i-1}: V^{\otimes m} \longrightarrow V^{\otimes m}, \quad (1\leq i\leq m-1).
\end{align}
Then it follows by Theorem~\ref{thm:YB} that $R_i\in \mathrm{End}_{\U}(V^{\otimes m})$ satisfy the braid relations:
\begin{equation}
  \label{eq:braidR}
R_iR_{i+1}R_i=R_{i+1}R_{i}R_{i+1}, \qquad R_iR_j=R_jR_i\quad (|i-j|>1).
\end{equation}

\section{Kazhdan-Lusztig bases for Hecke algebras}
  \label{sec:Hecke}

In this section, we will construct the Kazhdan-Lusztig (KL) bases for Hecke algebras. KL bases were introduced in the seminal  paper of Kazhdan-Lusztig \cite{KL79}. 

\subsection{Weyl groups}
Let $W$ be a finite Weyl group with identity $\mathbf{e}$; $W$ admits a Coxeter presentation
$$W=\left\langle s_i\; (i\in \I)~|~(s_is_j)^{m_{ij}}=\mathbf{e} \right\rangle,$$
where $m_{ij}=\mathrm{order}(s_is_j)$ in $W$, and $m_{ii}=1$.   There exists a length function $\ell: W\rightarrow \N$ such that $\ell(\sigma)=k$ if $\sigma\in W$ can be written as $\sigma=s_{i_1}s_{i_2}\cdots s_{i_k}$ with $k$ minimal; such $s_{i_1}s_{i_2}\cdots s_{i_k}$ is called a reduced expression of $\sigma$.

Denote by $w_0$ the longest element in $W$.
For example, the type $A_{m-1}$ Weyl group is the symmetric group $\mf S_m$:\
\[
\mathfrak{S}_m=\left\langle s_1,s_2,\ldots,s_{m-1}~|~s_i^2=\mathbf{e}, (s_is_{i+1})^3=\mathbf{e},
s_is_j=s_js_i\ (|i-j|>1) \right\rangle,
\]
where $s_i=(i,i+1)$ are simple transpositions. In this case, $w_0$ sends $i$ to $m+1-i$, for each $1\le i\le m$.

We say $\sigma$ is covered by $\sigma'$ ($\sigma\leftarrow\sigma'$) if $\sigma'=s_\beta\sigma$ for some $\beta\in\Phi$ and $\ell(\sigma')>\ell(\sigma)$. Define a partial order ``$<$'' on $W$ (called Bruhat order):
$\sigma<w$ if there exist $\beta_1,\ldots,\beta_k\in\Phi$ such that $w=s_{\beta_k}\cdots s_{\beta_1}\sigma$ and $\sigma\leftarrow s_{\beta_1}\sigma\leftarrow s_{\beta_2}s_{\beta_1}\sigma\leftarrow\cdots\leftarrow s_{\beta_k}\cdots s_{\beta_{1}}\sigma=w$.

\begin{lemma}
\label{lem:Bruhat}
  \begin{itemize}
    \item[(1)] If $\sigma$ is obtained from $w$ by deleting some elements in a reduced expression of $w$, then $\sigma<w$.
    \item[(2)]    $\sigma\leq w \; \Leftrightarrow \; w_0w\leq w_0\sigma.$
  \end{itemize}
\end{lemma}

\begin{example}
  Let $W=\mathfrak{S}_3=\langle s_1,s_2\rangle$. The Bruhat graph is as follows:
  \begin{equation}  \label{S3}
\xymatrix@R=0em@C=1em{
& s_1  \ar@{->}[ld]  & s_1s_2  \ar@{->}[l] \ar@{->}[ldd] &
  \\
\mathbf{e} &&& s_1s_2s_1=s_2s_1s_2 \ar@{->}[ld] \ar@{->}[lu]. \\
& s_2 \ar@{->}[lu]  & s_2s_1  \ar@{->}[l] \ar@{->}[luu] &
}
\end{equation}
  \end{example}

\subsection{Hecke algebras}
  \label{subsec:Hecke}

Let $\mathcal{A}=\Z[q,q^{-1}]$.
\begin{definition} \label{def:Hecke}
  The Hecke algebra $\H=\H_q(W)$ is a $\mathbb{Q}(q)$-algebra with unit $H_{\mathbf{e}}=1$, generated by $H_i, (i\in \I)$ subject to the following relations, for $i\neq j \in \I$,
  \begin{align*}
    (H_i+q)(H_i-q^{-1})=0;\qquad
    \underbrace{H_iH_jH_i\cdots}_{m_{ij}}=\underbrace{H_jH_iH_j\cdots}_{m_{ij}}.
  \end{align*}
\end{definition}

We list a few basic properties for the Hecke algebra below.
\begin{itemize}
  \item[(1)]
 Define  $H_{\sigma}=H_{i_1}\cdots H_{i_k}$, for any reduced expression $\sigma=s_{i_1}s_{i_2}\cdots s_{i_k}\in W$. Then $H_\sigma$ is independent of choices of reduced expressions of $\sigma$.
  \item[(2)]
  Denote by $\HA$ the $\A$-submodule of $\H$ spanned by
  $H_\sigma$, for $\sigma\in W$. Then $\HA$ is a free $\A$-module with basis  $\{H_\sigma~|~\sigma\in W\}$.
  \item[(3)]
  The $H_i$, for $i\in \I$,  is invertible with inverse $H_i^{-1}=H_i+(q-q^{-1})$.
  \item[(4)]
 $H_\sigma H_i=\left\{\begin{array}{ll}
    H_{\sigma s_i}, & \mbox{if $\ell(\sigma s_i)=\ell(\sigma)+1$};\\
    H_{\sigma s_i}+(q^{-1}-q)H_{\sigma}, &  \mbox{if $\ell(\sigma s_i)=\ell(\sigma)-1$}.
  \end{array}\right.$
\end{itemize}

The proof of the following lemma is elementary by checking relations.
\begin{lemma}
  There exists an automorphism $\overline{\phantom{x}}$ of the $\mathbb{Q}$-algebra $\H$ (respectively,  the $\Z$-algebra $\H_\A$) such that $\overline{q}=q^{-1}$ and $\overline{H_i}=H_i^{-1}$, for all $i\in \I$.
\end{lemma}

Note that $\overline{\bar{q}}=q$ and $\overline{\overline{H_i}}=\overline{H_{i}^{-1}}=\overline{H_i+(q-q^{-1})}=H_i^{-1}+(q^{-1}-q)=H_i$. Hence $\overline{\phantom{x}}$ is an automorphism of order $2$, which is called the bar involution of $\H$.

\subsection{Kazhdan-Lusztig basis}

\begin{theorem} [Kazhdan-Lusztig \cite{KL79}]
\label{thm:KL}
 For any $\sigma\in W$, there exists a unique element $C_\sigma\in\H$ such that \begin{itemize}
          \item[(1)]
           $\overline{C_\sigma}=C_{\sigma}$;
          \item[(2)]
           $C_\sigma\in H_\sigma+\sum_{w \in W} q\Z[q] H_w$.
        \end{itemize}
\end{theorem}
\begin{proof}

 \textbf{ Claim 1} (Existence). For any $\sigma\in W$, there exists a bar invariant element $C_\sigma\in H_{\sigma}+\sum_{w<\sigma}q\Z[q]H_w$.

 We prove the Claim by induction on $\sigma$ by the Bruhat order.
Let $\sigma=s_{i_1}\cdots s_{i_r}$ be a reduced expression. It is clear that $$\overline{H_\sigma}=H_{i_1}^{-1}\cdots H_{i_r}^{-1}=(H_{i_1}+(q-q^{-1}))\cdots(H_{i_r}+(q-q^{-1})),$$ and hence by Lemma~\ref{lem:Bruhat},
\begin{equation}\label{eq:Hbar}
\overline{H_\sigma}\in H_\sigma+\sum_{w<\sigma}\Z[q^{\pm1}]H_w.
\end{equation}
By the induction assumption, the above result implies
\begin{equation}\label{eq:9}
  \overline{H_\sigma}=H_\sigma+\sum_{\tau<\sigma}P_{\tau,\sigma}(q)C_\tau,
\end{equation}
and hence by applying the bar map,
\begin{equation}\label{eq:10}
  H_\sigma=\overline{H_\sigma}+\sum_{\tau<\sigma}P_{\tau,\sigma}(q^{-1})C_\tau.
\end{equation}
Comparing \eqref{eq:9} and \eqref{eq:10}, we have $P_{\tau,\sigma}(q^{-1})=-P_{\tau,\sigma}(q)$. Thus there exists unique $Q_{\tau,\sigma}(q)\in q\Z[q]$ such that $P_{\tau,\sigma}(q)=Q_{\tau,\sigma}(q)-Q_{\tau,\sigma}(q^{-1})$.
Define $$C_\sigma :=H_\sigma+\sum_{\tau<\sigma}Q_{\tau,\sigma}(q)C_\tau,$$ which clearly lies in $H_{\sigma}+\sum_{w<\sigma}q\Z[q]H_w$. It follows by \eqref{eq:9} that $\overline{C_\sigma}=C_\sigma$.

\vspace{2mm}
 \noindent
\textbf{ Claim 2} (Uniqueness).
For $H\in\sum_w q\Z[q]H_w$, if $\overline{H}=H$ then $H=0$.

Write $H=\sum h_wH_w$ with $h_w \in q\Z[q]$ and choose $\sigma$ maximal such that $h_\sigma\neq0$.
Then $H=\overline{H}$ and \eqref{eq:Hbar} imply that $h_\sigma=\overline{h_\sigma}$, which contradicts with $h_\sigma\in q\Z[q]$.

Now the uniqueness in the theorem follows from Claim 2.
\end{proof}

The existence established in the above proof is stronger than Theorem~\ref{thm:KL}(2), and will be recorded here.
\begin{corollary}  \label{cor:KL}
  For any $\sigma\in W$, it holds that $C_\sigma\in H_\sigma+\sum_{w<\sigma}q\Z[q]H_w$.
\end{corollary}
If follows by Corollary~\ref{cor:KL} that $\{C_\sigma~|~\sigma\in W\}$ forms an $\A$-basis for the Hecke algebra $\HA$ (or a $\mathbb Q(q)$-basis for $\H$); this is called the \emph{Kazhdan-Lusztig (KL) basis} of $\H$.
Write $C_\sigma=\sum_{w\leq\sigma} h_{w,\sigma}H_w$. Then $h_{w,\sigma}\in\Z[q]$ are called Kazhdan-Lusztig polynomials, with $h_{\sigma, \sigma}=1$.
Here are some examples of $C_\sigma$.

\begin{example}
  \begin{itemize}
    \item[(1)] $C_{s_i}=H_i+q$, for $i\in \I$.
    \item[(2)] ($q$-symmetrizer) Recall $w_0$ denotes the longest element in $W$. Then
    \[
    C_{w_0}=\sum_{\sigma\in W}q^{\ell(w_0)-\ell(\sigma)}H_\sigma,
    \]
    and $C_{w_0}H_i=q^{-1}C_{w_0}$. In particular, $\mathbb Q(q) C_{w_0}$ is a 1-dimensional ``trivial" submodule of the right regular representation $\H$. In Kazhdan-Lusztig theory the formula $C_{w_0}$ recovers the Weyl character formula.
  \end{itemize}
\end{example}
\subsection{Dual KL basis}

The same type of arguments for Theorem~\ref{thm:KL} leads to the following (where $q^{-1}$ plays the role of $q$ therein).
\begin{proposition}
\label{prop:dualKL}
There exists a unique element $L_\sigma \in \HA$, such that
\begin{itemize}
          \item[(1)] $\overline{L_\sigma}=L_{\sigma}$;
          \item[(2)] $L_\sigma\in H_\sigma+\sum_{w<\sigma} q^{-1}\Z[q^{-1}] L_w$.
        \end{itemize}
(The basis $\{L_\sigma~|~\sigma\in W\}$ is called the \emph{dual KL basis} of $\H$.)
\end{proposition}

\begin{example}
  \begin{itemize}
    \item[(1)] Here are some canonical and dual canonical basis elements for any (non-trivial) $W$:
    $$C_{s_i}=H_i+q, \quad C_{w_0}=\sum_{w\in W}q^{\ell(w_0)-\ell(w)}H_w,$$
    $$ L_{s_i}=H_i-q^{-1}, \quad L_{w_0}=\sum_{w\in W}(-q^{-1})^{\ell(w_0)-\ell(w)}H_w.$$
  \item[(2)]
  Take $W=\mathfrak{S}_3$ with Bruhat graph \eqref{S3}.  Then,
  \begin{align*}
  C_{s_1s_2}&=C_{s_1}C_{s_2}=H_{s_1s_2}+qH_1+qH_2+q^2,\\
   L_{s_1s_2}&=H_{s_1s_2}-q^{-1}H_1-q^{-1}H_2+q^{-2}.
\end{align*}
The formulas for $C_{s_2s_1}$ and $L_{s_2s_1}$ are similar.
  \end{itemize}
  \end{example}

\subsection{Parabolic KL basis}

Let $J\subset \I$. Define the parabolic subgroup
$$W_J:=\langle s_i~|~i\in J\rangle\leq W,$$
which is again a Coxeter/Weyl group. We denote by $w_{0,J}$ the longest element in $W_J$. The subalgebra $\H_J:=\langle H_j~|~j\in J\rangle$ of $\H$ is also the Hecke algebra associated to the Weyl group $W_J$. Let
\begin{align*}
C_J:=C_{w_{0,J}} &=\sum_{x\in W_J}q^{\ell(w_{0,J})-\ell(x)}H_x\in\H_J,
\\
L_J:=L_{w_{0,J}} &=\sum_{x\in W_J}(-q^{-1})^{\ell(w_{0,J})-\ell(x)}H_x \in\H_J.
\end{align*}
Denote by ${}^JW$ the set of minimal length right coset representatives for $W_J\backslash W$.

One can identify the right $\H$-module $C_J \H$ below with the right ``permutation $\H$-module"
\[
C_J \H \cong \Q(q) \otimes_{\H_J} \H,
\]
where $\mathbb Q(q)$ is the ``trivial" $\H_J$-module with $1\cdot H_j =q^{-1}$, for $j\in J$.
The following is a generalization of KL basis (which corresponds to the case $J=\emptyset$). We refer to \cite{So97} for an exposition of the following theorem.

\begin{theorem}[Deodhar]  \label{thm:Deo}
 For each $\sigma\in{}^JW$, there exists a unique element $C_\sigma^J\in C_J \H$ such that
 \begin{itemize}
   \item[(1)]
   $\overline{C_\sigma^J} =C_\sigma^J$;
   \item[(2)]
   $C_\sigma^J\in C_J(H_\sigma+\sum_{w\in{}^JW, w<\sigma}q\Z[q]H_w)$.
 \end{itemize}
 Hence $\{C_\sigma^J~|~ \sigma\in{}^JW \}$ forms a basis for $C_J \H$ (called parabolic KL basis). 
\end{theorem}

  Similarly, there exists a \emph{parabolic dual KL basis} $\{L_\sigma^J | \sigma\in{}^JW\}$ such that
  \begin{itemize}
 \item[($1'$)] $\overline{L_\sigma^J} = L_\sigma^J$;
 \item[($2'$)] $L_\sigma^J\in L_J(H_\sigma+\sum_{w\in{}^JW, w<\sigma}q^{-1}\Z[q^{-1}]H_w).$
 \end{itemize}

\section{Canonical bases arising from quantum groups}
  \label{sec:CB}

In this section, we construct canonical bases for $\U^-$, finite-dimensional simple $\U$-modules, their tensor products and modified quantum groups. The quasi R-matrix is used to define the bar involution on these tensor products. Canonical bases for $\U^-$ and finite-dimensional simple $\U$-modules have been widely known. We refer to the books \cite{Lus93, Jan95, DDPW}; for an elementary approach, see \cite{Tin16}. For based $\U$-modules and canonical bases on tensor products, see \cite{Lus92, Lus93}.

\subsection{Canonical basis on $\U^{-}$}
  \label{subsec:CBU}

We assume in this subsection that $\U$ is of ADE type; see however Remark~\ref{rem:KM}. Recall $\A =\Z[q,q^{-1}]$.
The $\A$-subalgebra $\UA^{-}$ of $\U^{-}$ 
 is generated by the divided powers $F_i^{(r)}$, for $i\in \I, r\geq0$.

\begin{remark}
The $\A$-algebra $\UA^{-}$ is a free $\A$-module, and $\U^-=\mathbb{Q}(q)\otimes_{\A}\UA^{-}$. The algebra $\UA^{-}$ admits a (Ringel-)Hall algebra realization; cf. \cite{DDPW}.
\end{remark}

The canonical basis of $\UA^-$ was first constructed by Lusztig \cite{Lus90}.
We outline the main steps toward the canonical basis of $\UA^-$ using the Hall algebra approach \cite{DDPW}.

(1) In the Hall algebra realization of quantum groups, the braid group symmetries $T_i\in\mathrm{Aut}(\U)$, for $i\in \I$ are realized as reflection functors. This gives rise to algebra isomorphisms on some proper subalgebras of $\U^-$, which leads to a construction of $q$-root vectors $F_\beta\in\UA^{-}$, for $\beta\in\Phi^+$, and their realization as indecomposable modules in the Hall algebra.

(2) A suitable sequence of sinks/sources on the Dynkin quiver gives rise to a total ordering on $\Phi^+$; it corresponds to some distinguished reduced expression for the longest element $w_0$ of the Weyl group $W$.

(3) Applying the braid group symmetries to the divided powers $F_i^{(a)}$ gives us the divided powers $F_\beta^{(a)}$, for $\beta\in \Phi^+$.
For any $\underline{n}=(a_\beta)_{\beta\in\Phi^+}\in\N^{\Phi^+}$, define
            $$
            F^{(\underline{n})}:= \overrightarrow{\prod}_{\beta\in\Phi^+}F_\beta^{(a_\beta)} \in\UA^{-}.
            $$
            Then $\left \{F^{(\underline{n})}~\middle|~\underline{n}\in\N^{\Phi^+} \right\}$ forms a PBW basis for $\UA^{-}$. The PBW basis has a Hall algebra interpretation.

(4) For $\underline{n}=(a_\beta)_{\beta\in\Phi^+}$, we denote
\[
||\underline{n}||:=\sum_{\beta\in\Phi^+}a_\beta \beta.
\]
Degeneration of modules of a Dynkin quiver (which corresponds to orbit closure relation) gives rise to a partial order ``$\leq$'' on $\left\{\underline{n}=(a_\beta)_{\beta\in\Phi^+}~\middle|~||\underline{n}|| =\mu \right\}$ for any fixed $\mu\in\mathbb{N}\Phi^+$.

(5) Recall the bar involution ${}^{-}$ on $\U$ preserves $F_i$ and the divided powers $F_i^{(r)}$ (because of $\overline{[r]_i}=[r]_i$), and hence, it preserves the $\A$-algebra $\UA^{-}$.

(6) We have
\begin{equation}\label{eq:12}
          \overline{F^{(\underline{n})}}\in F^{(\underline{n})}+\sum_{\underline{m}<\underline{n}}\A F^{(\underline{m})}.
        \end{equation}
A monomial basis construction is used to ensure that the coefficient of $F^{(\underline{n})}$ is $1$ in \eqref{eq:12}.

\begin{theorem} [Lusztig]
  \label{thm:CBade}
  Let $\U$ be of ADE type.
For $\underline{n}\in\N^{\Phi^+}$, there exists a unique element $b_{\underline{n}}\in\UA^{-}$ such that
        $$ \overline{b_{\underline{n}}}=b_{\underline{n}},\quad b_{\underline{n}}\in F^{(\underline{n})}+\sum_{\underline{m}}q\Z[q]F^{(\underline{m})}.$$
        Moreover, $ b_{\underline{n}}\in F^{(\underline{n})}+\sum_{\underline{m}<\underline{n}}q\Z[q]F^{(\underline{m})}$.
        The set $B:=\{b_{\underline{n}}~|~\underline{n}\in\mathbb{N}^{\Phi^+}\}$ forms an $\A$-basis for $\UA^-$ (called the canonical basis).
  \end{theorem}
\begin{proof}
 We have done all the necessary preparations in Steps (1)-(6) above.  The theorem now follows from \eqref{eq:12}, by exactly the same argument as in the proof of Theorem~\ref{thm:KL} for KL basis starting from the equation \eqref{eq:9} (which is the Hecke counterpart of \eqref{eq:12}).
\end{proof}

\begin{remark}  \label{rem:KM}
The original approach toward canonical basis on $\U^-$ based on monomial basis and PBW bases was given in \cite{Lus90}.
An elementary self-contained proof of Theorem~\ref{thm:CBade} in ADE type can be found in \cite{Tin16}.

The above theorem is actually valid for all finite type and even for Kac-Moody type, cf. \cite{Lus90, Ka91, Lus93} for different approaches. We shall assume this for all finite type below.
\end{remark}

\begin{example} The canonical basis in $\U_q^{-}(\mathfrak{sl}_3)$ is
  \[
  \{F_1^{(a)}F_2^{(b)}F_1^{(c)}, F_2^{(a)}F_1^{(b)}F_2^{(c)}~|~b\geq a+c\},
  \]
  modulo the identification $F_1^{(a)}F_2^{(b)}F_1^{(c)}=F_2^{(c)}F_1^{(b)}F_2^{(a)}$ if $b=a+c$.

 For example, we can write $F_1 F_2 F_1$ (a product of 2 canonical basis elements) as a (positive) linear combination of canonical basis elements $F_1 F_2 F_1 =F_1^{(2)} F_2 + F_2 F_1^{(2)}$, thanks to the $q$-Serre relation. This example also exhibits a deep positivity property of canonical basis \cite{Lus93}, that is, the structure constants of $B$ lie in $\N[q,q^{-1}]$.
%
\end{example}
\subsection{Canonical basis on simple $\U$-modules}
  \label{subsec:CBM}

For $\lambda\in X^+$, let $M(\lambda)$ be the Verma $\U$-module of highest weight $\lambda$ and $L(\lambda)$ its finite dimensional irreducible quotient $\U$-module. Denote by $v^+_\lambda$ the highest weight vector in $L(\lambda)$. Define $L(\lambda)_{\A} :=\UA^{-}\cdot v^+_\lambda$. We list several well-known properties for the $\A$-form $L(\lambda)_{\A}$ below:
  \begin{itemize}
    \item[(1)] $L(\lambda)_{\A}$ is a $\UA$-module;
    \item[(2)] $L(\lambda)_{\A}$ is a free $\A$-module, and $L(\lambda)=\mathbb{Q}(q)\otimes_{\A}L(\lambda)_{\A}$;
    \item[(3)] $L(\lambda)_{\A}$ (or $L(\lambda)$) admits a bar involution such that $\overline{u\cdot v^+_\lambda}=\overline{u}\cdot v^+_\lambda$ for any $u\in\UA^{-}$.
  \end{itemize}

Recall $B$ is the canonical basis for $\U^-$.
\begin{theorem}[Lusztig]
   Let $\lambda\in X^+$. Then $B(\lambda):=\{bv^+_\lambda~|~b\in B\}\setminus\{0\}$ forms a basis for $L(\lambda)_{\A}$ or $L(\la)$ (called a canonical basis). Moreover, the map $\{b\in B~|~bv^+_\lambda\neq0\}\rightarrow B(\lambda), b\mapsto bv^+_\lambda$, is a bijection.
\end{theorem}

\begin{example}
  Let $\V$ be the natural representation of  $\U_q(\mathfrak{sl}_N)$ with a standard basis $\{v_1,\ldots, v_N\}$. The canonical basis for $\V$ coincides with its standard basis, i.e., $B(\epsilon_1)=\{v_1,\ldots,v_N\}$.
\end{example}

\subsection{Based $\U$-modules}
 \label{subsec:based}

A $\U$-module $M$ equipped with a bar involution $\overline{\phantom{x}}$ is called {\em involutive} if the bar involution on $M$ is compatible with the bar involution on $\U$, i.e., $\overline{u\cdot x}=\overline{u}\cdot\overline{x}$ for all $u\in\U, x\in M$.

A {\em based $\U$-module} $(M,B)$ consists of an involutive weight $\U$-module $M =\oplus_\mu M^\mu$ with a bar invariant basis $B$, such that $B \cap M^\mu$ is a basis for $M^\mu$ (for all $\mu$), $M_{\A}:=\A B$ is a $\U_\A$-module, $M_{\Z[q]}:=\Z[q]B$ is a $\Z[q]$-lattice of $M$, and the image of $B$ forms a basis for $M_{\Z[q]}/qM_{\Z[q]}$; cf. \cite[27.1.1]{Lus93}.
For example, $(L(\lambda),B(\lambda))$ is a based $\U$-module.

Given finite dimensional based $\U$-modules $M, M'$, the ``naive'' bar map $\overline{\cdot\otimes\cdot} =\overline{\cdot}\otimes\overline{\cdot}$ on the $\U$-module $M\otimes M'$ is not involutive in general since
\[
\overline{u\cdot(x\otimes x')}=\overline{\Delta(u)(x\otimes x')}=\overline{\Delta(u)}(\overline{x}\otimes \overline{x'})\neq \Delta(\overline{u})(\overline{x}\otimes \overline{x'})=\overline{u}\cdot(\overline{x}\otimes \overline{x'}),
\]
for all $u\in \U, x\in M, x' \in M'$.

Recall the quasi R-matrix $\Theta\in\widehat{\U\otimes \U}$ from Theorem~\ref{thm:quasi}.
\begin{lemma}
 \label{lem:barMN}
Let $M, M'$ be finite dimensional based $\U$-modules. Then the anti-linear map
\[
\psi=\Theta\circ(\overline{\cdot}\otimes\overline{\cdot}),
\]
 i.e., $\psi(x\otimes x')=\Theta(\overline{x}\otimes \overline{x'})$, makes $M\otimes M'$ an involutive $\U$-module.
\end{lemma}
(We shall call $\psi$ the bar involution on $M\otimes M'$.)

\begin{proof}
By Theorem~\ref{thm:quasi} and Corollary~\ref{cor:Theta}, we have
\[
\psi^2(x\otimes x')=\Theta\circ(\overline{\cdot}\otimes\overline{\cdot})\circ\Theta(\overline{x}\otimes \overline{x'})=\Theta\overline{\Theta}(x\otimes x')=x\otimes x'.
\]
Moreover, the bar involution $\psi$ is compatible with the bar involution on $\U$, that is,
\[
\psi(u\cdot(x\otimes x'))=\Theta\overline{\Delta(u)}(\overline{x}\otimes \overline{x'})=\Delta(\overline{u})\Theta(\overline{x}\otimes \overline{x'})=\overline{u}\cdot\psi(x\otimes x'),
\]
for all $u\in \U, x\in M, x' \in M'$. Therefore, $(M\otimes M',\psi)$ is an involutive $\U$-module.
\end{proof}

Given based $\U$-modules $M_1, \ldots, M_m$, the tensor product $M_1\otimes \ldots \otimes M_m$ is an involutive $\U$-module by applying Lemma~\ref{lem:barMN} repeatedly (the resulting bar involution is independent of the choices of the bracketing on the tensor factors).

\subsection{Canonical basis on tensor products}

Let $(M,B), (M', B')$ be based $\U$-modules. Define a partial order $\leq$ on $B\times B'$ by
$$(b_2,b_2')\leq(b,b')\; \Leftrightarrow \; \mathrm{wt}(b)+\mathrm{wt}(b')=\mathrm{wt}(b_2)+\mathrm{wt}(b_2') \ \mbox{and}\ \mathrm{wt}(b_2) \geq \mathrm{wt}(b).$$

\begin{theorem} \cite{Lus93}
 \label{thm:Ltensor}
  \begin{itemize}
    \item[(1)] For any $(b,b')\in B\times B'$, there exists a unique $b\diamond b'\in\Z[q](B\times B')$ such that $\psi(b\diamond b')=b\diamond b'$ and $b\diamond b'\in b\otimes b'+q\Z[q](B\times B')$.
    \item[(2)] $b\diamond b'\in  b\otimes b'+\sum_{(b_2,b_2')<(b,b')}q\Z[q]b_2\otimes b_2'$.
    \item[(3)] $B \diamond B' :=\{b\diamond b'~|~(b,b')\in B\times B'\}$ forms a basis for $M\otimes M'$ (also for $M_{\A}\otimes M'_{\A}$ and for $\Z[q] B\otimes B'$).
  \end{itemize}
\end{theorem}
The basis $B \diamond B'$ is called the canonical basis for $M\otimes M'$.
It follows that $(M\otimes M', B \diamond B')$ is a based $\U$-module.

\begin{proof}
Note $\psi (b\otimes b') = \Theta (b\otimes b')$, for $(b,b')\in B\times B'$. Thanks to the fact that $\Theta$ lies in (a completion of) $\UA^+\otimes \UA^-$ \cite{Lus93}, we have
\[
\psi (b\otimes b')  \in b\otimes b'+\sum_{(b_1,b_1')<(b,b')}\A \, b_1\otimes b_1'.
\]
The rest of the argument is standard; see the proof of Theorem~\ref{thm:KL}.
\end{proof}

\begin{corollary}   \label{cor:CBtensor}
  Let $\lambda^1,\ldots,\lambda^k\in X^+$. Then $L(\lambda^1)\otimes\cdots\otimes L(\lambda^k)$ admits a canonical basis.
\end{corollary}

\subsection{Canonical basis on modified quantum groups}
\label{subsec:CBQG}

To extend the theory of canonical basis from half a quantum group $\U^-$ to the whole quantum group, we need the notion of modified quantum groups \cite{BLM90, Lus93}, where one replaces the Cartan subalgebra $\U^0$ of $\U$ by a family of orthogonal idempotents ${\bf 1}_\zeta$, for $\zeta\in X$:
\[
\dot\U = \bigoplus_{\zeta \in X} \dot\U {\bf 1}_\zeta = \bigoplus_{\zeta \in X} \U^-\U^+ {\bf 1}_\zeta.
\]
We refer to \cite[IV]{Lus93} for an elementary definition of $\dot\U$.

For $\la \in X^+$, recall $L(\la)$ denotes the irreducible $\U$-module with highest weight $\la$ and a highest weight vector $v^+_\la$.  Let ${}^\omega L(\la)$ denote the $\U$-module with the same underlying vector space as $L(\la)$ but with action of $\U$ twisted by the Chevalley involution $\omega$; the vector $v^+_\la$ becomes a lowest weight vector  in ${}^\omega L(\la)$ and will be denoted by $v^-_{-\la}$. In other words, ${}^\omega L(\la)$ is the lowest weight $\U$-module of lowest weight $-\la$, which is isomorphic to $L(-w_0 \la)$ of highest weight $-w_0 \la \in X^+$.

Given $\zeta \in X$, we can choose and fix $\la^1, \la^2 \in X^+$ such that $\la^1 -\la^2 =\zeta$. For each $\nu \in X^+$, we note that $L(\la^1+\nu) \otimes {}^\omega L(\la^2+\nu)$ is a cyclic $\U$-module generated by $v^+_{\la^1+\nu} \otimes v^-_{-\la^2-\nu}$, and its canonical basis is defined by Corollary~\ref{cor:CBtensor}. Following Lusztig \cite{Lus93}, sending $v^+_{\la^1+\nu} \otimes v^-_{-\la^2-\nu}$ to $v^+_{\la^1} \otimes v^-_{-\la^2}$ (both of weight $\zeta$) defines a surjective $\U$-module homomorphism
\[
\pi_{\la^1,\la^2, \nu}: L(\la^1+\nu) \otimes {}^\omega L(\la^2+\nu) \longrightarrow L(\la^1) \otimes {}^\omega L(\la^2).
\]
Indeed, the homomorphism $\pi_{\la^1,\la^2, \nu}$ arises as a composition of the following $\U$-module homomorphisms from \cite[\S25.1.1]{Lus93} including a contraction map $\delta_\nu: L(\nu) \otimes {}^\omega L(\nu) \rightarrow \Q(q)$:
\begin{align} \label{xy:pi}
  \xymatrix{
 L(\la^1+\nu) \otimes {}^\omega L(\la^2+\nu)     \ar[dr]_{\pi_{\la^1,\la^2, \nu}} \ar[r]
              & L(\la^1)\otimes L(\nu) \otimes {}^\omega L(\nu) \otimes {}^\omega L(\la^2)  \ar[d]^{1 \otimes \delta_\nu \otimes 1}  \\
             &   L(\la^1) \otimes {}^\omega L(\la^2)
                }
\end{align}
Remarkably, $\pi_{\la^1,\la^2, \nu}$ is a homomorphism of based modules, i.e., sending a canonical basis element to a canonical basis element or 0. Hence, we obtain a projective system of $\U$-modules $\big\{ L(\la^1+\nu) \otimes {}^\omega L(\la^2+\nu) \big\}_{\nu \in X^+}$ with compatible canonical bases. This leads to a canonical basis on its inverse limit, while the inverse limit can be naturally identified with $\dot \U {\bf 1}_\zeta$.

In this way, the canonical basis on tensor products leads to a construction of the canonical basis on $\dot\U$.
See \cite[IV]{Lus93} for more details.

\section{$q$-Schur duality}
  \label{sec:qSchur}

In this section, we will establish the $q$-Schur duality (or Jimbo-Schur duality) between quantum group and Hecke algebra of type A, which goes back to \cite{Jim86}. The R-matrix gives rise to an action of the Hecke algebra on the tensor power $\V^{\otimes m}$ of the natural representation of $\U$. The type A Kazhdan-Lusztig basis and canonical basis on the tensor power coincide; cf. \cite{FKK98}.

\subsection{$q$-Schur duality of type A}
\label{sub:Hecke}

Given $a<b \in \mathbb R$ with $b-a \in \N$, introduce the ``integral interval" notation:
\begin{align}  \label{eq:interval}
[a..b] := \big\{a, a+1, a+2,\ldots,b \big \}.
\end{align}
The $q$-tensor space $\V^{\otimes m}$ has a basis
\[
\Big\{M_f:=v_{f(1)}\otimes \cdots \otimes v_{f(m)}~|~f: [1..m] \rightarrow [1..N] \Big\}.
\]
We identify  $f: [1..m] \rightarrow [1..N]$ with a $m$-tuple $(f(1), \ldots, f(m)) \in [1..N]^m$.

In this subsection, we restrict ourselves to the type A  Hecke algebra $\H =\H_q(\mf S_m) =\langle H_i~|~1\le i \le m-1 \rangle$; cf. Definition~\ref{def:Hecke}. In particular, $(H_i+q)(H_i-q^{-1})=0$.

The symmetric group $\mf S_m$ acts on $[1..N]^m$ on the right by $(f \sigma) (i) =f(\sigma i)$, for $1\le i \le m$. The following lemma can be verified by a direct computation.

\begin{lemma} \label{lem:HtypeA}
  A (right) action of the Hecke algebra $\H =\H_q(\mf S_m)$ on $\V^{\otimes m}$ is given by, for $1\le i\le m-1$,
  $$M_fH_i=\left\{\begin{array}{ll}
    M_{fs_i}, & \mbox{if $f(i)<f(i+1)$};\\
    M_{fs_i}+(q^{-1}-q)M_f, & \mbox{if $f(i)>f(i+1)$};\\
    q^{-1}M_f, & \mbox{if $f(i)=f(i+1)$}.
  \end{array} \right.$$
\end{lemma}

As $\U$ admits a natural triangular decomposition $\U =\U^- \U^0 \U^+$, the highest weight theory for $\U$-modules makes sense, and we denote by $L_q(\la)$ the irreducible $\U$-module (generated by a highest weight vector $v^+_\la$) of highest weight $\la \in \mathcal{P}_m(N)$, where $D_i v^+_\la =q^{\la_i} v^+_\la$, for $1\le i \le N$. It is known that $L_q(\la)$ shares the same character as the usual $\mathfrak{gl}_N$-module $L(\la)$; cf. \cite{Lus93, Jan95}.

On the other hand, the Specht module $S_q^\la$ over $\H$, for $\la \in \mathcal{P}_m$, can be defined (cf. \cite{Mat99}); they are a deformation of the Specht module $S^\la$ and in particular have the same dimension.

\begin{theorem} \label{thm:Jimbo}
\cite{Jim86} {\quad}
\begin{itemize}
  \item[(1)] The actions
  \[
\U=\U_q(\mathfrak{gl}_N)\stackrel{\Phi}{\curvearrowright} \V^{\otimes m}\stackrel{\Psi}{\curvearrowleft} \H_q(\mathfrak{S}_m)=\H
\]
commute, and they form double centralizers:
  \begin{align*}
    \Phi(\U)= &\mathrm{End}_{\H} (\V^{\otimes m}),\\
    &\mathrm{End}_{\U}(\V^{\otimes m})^{\mathrm{op}} =\Psi(\H).
  \end{align*}
  \item[(2)] As a $(\U,\H)$-module, $\V^{\otimes m}\cong \bigoplus_{\lambda\in\mathcal{P}_m(N)}L_q(\lambda)\otimes S_q^\lambda$.
\end{itemize}
\end{theorem}
Let us sketch a proof of the above theorem. We can establish (2) above first. By comparing with its classical counterpart in Theorem~\ref{schurduality2}, we know the composition factors of the $\U$-module $\V^{\otimes m}$ are as indicated with the expected multiplicity spaces (which become irreducible $\H$-modules). Part (1) follows from (2).

The finite-dimensional $\Q(q)$-algebra
\[
S_q(N,m) :=\mathrm{End}_{\H} (\V^{\otimes m}) =\Phi(\U)
\]
is called the {\em $q$-Schur algebra}. There are different (algebraic and geometric) approaches  to the $q$-Schur duality using the $q$-Schur algebra $\mathrm{End}_{\H} (\V^{\otimes m})$ developed in \cite{DJ89} and in \cite{BLM90}; these approaches have been generalized to arbitrary finite type \cite{LW20}.

\subsection{Jimbo-Schur duality via R-matrix}

Recall the action of R-matrix $R_i$ on the tensor power $\V^{\otimes m}$ from \eqref{eq:Ri}. The following theorem improves the $q$-Schur duality in Theorem~\ref{thm:Jimbo}.

\begin{theorem} \cite{Jim86}
 \label{thm:Jimbo2}
  Let $\U=\U_q(\mathfrak{sl}_N)$ or $\U_q(\mathfrak{gl}_N)$ and $\V$ its natural representation. Then the action of $\H$ on $\V^{\otimes m}$ (see Lemma~\ref{lem:HtypeA}) is realized via R-matrix: $H_i\mapsto R_i^{-1}$, for $1\le i \le m-1$.
\end{theorem}

\begin{proof}
As $R_i^{-1}$ already satisfies the braid relations \eqref{eq:braidR},
it suffices to restrict to the case when $m=2$, $R_1=R$ and to verify $H_1^{-1} =R|_{\V\otimes \V}$.

  The $\Theta =\Theta_\alpha$ for $\mathfrak{sl}_2$ is given by \eqref{eq:dual}, and this gives us
  \begin{align}
    \Theta_\alpha|_{\V\otimes \V}
    &=1\otimes 1+\sum_{n\geq1}  (-q)^{-n} (\theta^n, \theta^n)^{-1}E_{\alpha}^n\otimes F_{\alpha}^n|_{\V\otimes \V}
    \notag \\
    &=1\otimes 1-(q^{-1}-q)E_\alpha\otimes F_\alpha|_{\V\otimes \V}. \label{eq:7}
  \end{align}

 The $\Theta$ admits a factorization as an ordered product over $\alpha\in\Phi^+:\Theta=\overrightarrow{\prod}_{\alpha\in\Phi^+}\Theta_\alpha$. From this one shows that $\Theta|_{\V\otimes \V}=1\otimes 1-(q^{-1}-q)\sum_{i<j}E_{ij}\otimes E_{ji}$ where any higher power vanishes.
     (Here $E_{ij}$ are $q$-root vectors, but can be identified with classical matrix units $E_{ij}$ when acting on the standard basis of $\V$.)

  Choose $f$ in Lemma~\ref{lem:2} with
  \begin{align} \label{ftilde}
  f(\epsilon_i,\epsilon_i)=q, \quad f(\epsilon_i,\epsilon_j)=1 \; (\text{for } i \neq j).
  \end{align}
  Then
 \[
 \Theta^f|_{\V\otimes \V}=q\sum_{i}E_{ii}\otimes E_{ii}+\sum_{i\neq j}E_{ii}\otimes E_{jj}-(q^{-1}-q)\sum_{i<j}E_{ij}\otimes E_{ji}.
 \]
Now it is straightforward to check $H_1^{-1} =R|_{\V\otimes \V}$ via $R =\Theta^f\circ P$.
\end{proof}

\subsection{Canonical basis vs KL basis of type A}

In this subsection, again $\U =\U_q(\mf{sl}_N)$ and $\H =\H_q(\mf S_m)$.
Recall $\V$ is the natural representation of $\U$ with basis $\{v_1, \ldots, v_N\}$.
Let $\VA =\A\langle v_1,\ldots,v_N\rangle$ be the $\A$-submodule of $\V$.
Write
\[
M_f=v_{f(1)}\otimes \cdots\otimes v_{f(m)},
\quad \text{ for } f\in [1..N]^m.
\]
We shall refer to
$\{M_f~|~f\in[1..N]^m\}$ as the {\em standard basis} for $\V_\A^{\otimes m}$.

Let
\[
\Lambda^-(N,m) =\{f\in[1..N]^m~|~f(1)\leq \cdots \leq f(m)\}
\]
be the set of {\em anti-dominant} weights (with respect to $\U$).
We have the following isomorphism of $\H$-modules:
\begin{equation}\label{eq:11}
  \V_\A^{\otimes m}\cong\bigoplus_{f\in\Lambda^-(N,m)} C_{J(f)}\H,
  \quad M_f\mapsto C_{J(f)}
\end{equation} where $J(f)=\{j~|~fs_j=f\}$. The parabolic KL bases on $C_{J(f)}\H$, for various $f$, induce via the isomorphism \eqref{eq:11} a basis for $\V^{\otimes m}$, which we shall refer to as a {\em KL basis}.

On the other hand, by Lemma~\ref{lem:barMN} and Corollary~\ref{cor:CBtensor}, as a $\U$-module $\V^{\otimes m}$ admits a bar invlution $\psi$ and a canonical basis.

\begin{lemma}
 \label{lem:3bar}
The bar involution $\psi$ on $\V^{\otimes m}$ as a $\U$-module satisfies that
\[
    \psi (uxh)=\psi (u)\psi (x)\overline{h}, \quad \text{ for } u\in\U, x\in\V^{\otimes m}, h\in\H.
\]
Moreover, such a bar involution is unique by requiring that $\psi (M_f)=M_f$ for all $f \in \Lambda^-(N,m)$, i.e., for all  anti-dominant $f$.
\end{lemma}

\begin{proof}
By definition, the action of $\psi$ on $\V^{\otimes m}$ is defined by iteration via $\psi =\Theta (\psi \otimes \psi)$. Since $\Theta =\sum_{\mu \in \N \I} \Theta_\mu$ with $\Theta_\mu \in \U^+_\mu \otimes \U^-_{\mu}$ by Theorem~\ref{thm:quasi}, we see inductively that $\psi (M_f) \in M_f +\sum_{g} \Q(q) M_g$, summed over $g$ such that $M_g$ has the same $\U$-weight as $M_f$ and $g<f$ in the Bruhat ordering of $\mf S_m$. It follows that $\psi (M_f) =M_f$, for $f$ anti-dominant.

The uniqueness is clear as $\V^{\otimes m}$ is a direct sum of permutation modules $M_f  \H$, summed over all anti-dominant $f$.

We already know that $\V^{\otimes m}$ is an involutive $\U$-module, i.e., $\psi (ux)=\psi (u)\psi (x)$.
It remains to check that
\begin{equation}
 \label{mHi}
\psi (x H_i)= \psi (x) H_i^{-1}, \quad \text{ for all } x \in\V^{\otimes m},  1\le i \le m-1.
\end{equation}

The proof of \eqref{mHi} proceeds by induction on $m$. One first checks the case when $m=2$ by using $H_i^{-1} =\Theta \widetilde f P$ as follows. One computes that $\psi(x H_1) =\Theta (\overline{x} \overline{H_1}) =\Theta (\overline{\Theta} \overline{\widetilde f}  P)^{-1} (\overline{x}),$
and $\psi(x) H_1^{-1} = (\Theta \widetilde f  P) ( \Theta (\overline{x}) ).$
They are equal thanks to $\overline{\Theta}^{-1} =\Theta$ and $(\overline{\widetilde f}  P)^{-1} =P \overline{\widetilde f}^{-1} =P \widetilde f =\widetilde{f} P$; recall $f$ from \eqref{ftilde} is symmetric.

Assume that \eqref{mHi} holds for $(m-1)$th tensor, for $m\ge 3$. Then, for $2\le i \le m-1$, denoting $x = v\otimes x'$, for $v\in \V$, we verify  \eqref{mHi} by
\begin{align*}
\psi(x H_i) &= \Theta (\psi(v) \otimes \psi(x' H_i))
\\
&= \Theta (\psi(v) \otimes \psi(x') \overline{ H_i})
\\
& = \Theta (\psi(v) \otimes \psi(x')) \overline{ H_i}
= \psi(x) \overline{ H_i},
\end{align*}
where the second equality above uses the inductive assumption and the third equality uses $q$-Schur duality (and that $\Theta$ lies in a completion of $\U\otimes \U$).
For $H_i$ with $i=1$ (or more generally, for $1\le i \le d-2$), a similar computation with $x =x'\otimes v$ for $v\in \V$ verifies \eqref{mHi}. The lemma is proved.
\end{proof}

\begin{theorem}  (cf. \cite[Theorem 2.5]{FKK98})
  \label{thm:CBKLa}
  The canonical (and respectively, dual canonical) basis and KL (and respectively, dual KL) basis on $\V^{\otimes m}$ coincide.
\end{theorem}

\begin{proof}
We only explain the identification ``canonical =KL", as the dual version is similar.
By Lemma~\ref{lem:3bar}, we can think the bar involution on $\V^{\otimes m}$ from either the Hecke algebra or quantum group viewpoint. Each KL basis element on $\V^{\otimes m}$ is bar invariant, and of the form $M_f + \sum_g q \Z[q] M_g$ (cf. Theorem~\ref{thm:Deo}). By the uniqueness of the canonical basis (cf. Theorem~\ref{thm:Ltensor}), it must be a canonical basis element.
\end{proof}

\begin{example}
  \begin{itemize}
    \item[(1)]
    The canonical basis $\{v_1,\ldots,v_N\}$ for $\V$ coincides with its standard basis.
    \item[(2)]
    The canonical basis for $\V^{\otimes 2}$ is
    \[
    \{v_a\otimes v_b, \quad v_a\otimes v_a, \quad v_b\otimes v_a+qv_a\otimes v_b~|~a<b\}.
    \]
  \end{itemize}
\end{example}

%

\section{$\imath$Quantum groups and $\imath$-canonical basis}
\label{sec:iQG}

The $\imath$quantum groups arising from quantum symmetric pairs are introduced, and $\imath$-canonical basis on $\U$-modules regarded as $\Ui$-modules is formulated. A key ingredient is the formulation and construction of the quasi K-matrix.

\subsection{Quantum symmetric pairs and $\imath$Quantum groups}

The theory of quantum symmetric pairs of finite type was systematically formulated by G.~Letzter \cite{Let99}, though examples of $\imath$quantum groups (of classical type) were known earlier. It has been extended to Kac-Moody setting by Kolb \cite{Ko14}. Unlike Letzter's papers, Kolb's formulation of quantum symmetric pairs and $\imath$quantum groups is compatible with Lusztig's book, which makes it a standard working reference in the field.

Recall we are given a simple system $\I$ of finite type,  a weight lattice $X$, and a bilinear form $\cdot: \Z \I \times X \rightarrow \Z$.
We call a permutation $\tau$ of the set $\I$ an {\em involution} of the Cartan datum $(\I, \cdot)$ if $\tau^2 =\mathrm{id}$ and $\tau i \cdot \tau j = i \cdot j$ for $i$, $j \in \I$.  Note we allow $\tau =\mathrm{id}$.
The permutation $\tau$ of $\I$ induces a $\Q(q)$-algebra automorphism of $\U$,
defined by
\begin{align}
 \label{eq:tau}
\tau:E_i \mapsto E_{\tau i}, \quad F_i \mapsto F_{\tau i}, \quad K_i \mapsto K_{\tau i}, \qquad \mbox{for all $i\in \I$}.
\end{align}

\begin{remark}
 \label{rem:QSP}
Let us briefly put the $\imath$quantum groups in perspective.
  \begin{itemize}
    \item[(1)]
    A symmetric pair $(\mathfrak{g},\mathfrak{g}^{\theta})$ consists of a simple Lie algebra (over $\mathbb{C}$) and its fixed point subalgebra, where $\theta$ is an involution on $\mathfrak{g}$. There is a one-to-one correspondence between the classification of symmetric pairs and the classification of real forms of simple Lie algebras over $\mathbb{C}$; they can be classified by Satake diagrams, $(\I =\Iw \cup \Ib, \tau)$, which consists of a bicolor partition of the Dynkin diagram vertex set $\I$ and a diagram involution $\tau$ (which is allow to be $\mathrm{id}$), subject to some constraints.
    \item[(2)]
A quantum symmetric pair (QSP) $(\U,\Ui)$, as formulated by G.~Letzter  (see \cite{Let99}), consists of a quantum group $\U$ and a coideal subalgebra $\Ui$ of $\U$ (i.e., $\Delta: \Ui \longrightarrow \Ui \otimes \U$), such that the $q\mapsto 1$ limit of $(\U,\Ui)$ becomes $(U(\mathfrak{g}),U(\mathfrak{g}^\theta))$. We call $\Ui$ an {\em $\imath$quantum group} ($\imath$QG).
An $\imath$QG is called {\em quasi-split} if $\I=\I_\circ$. 
    \item[(3)]
A quantum group is an $\imath$quantum group (just as a complex Lie group can be viewed as a real Lie group). Indeed, 
$(\U\otimes\U,\U)$ forms a QSP of diagonal type, where the embedding $(\omega \otimes  1) \circ\Delta:\U\rightarrow\U\otimes\U$ makes $\U$ a coideal subalgebra of $\U\otimes \U$, and $\omega$ is the Chevalley involution in Lemma~\ref{lem:inv}.
  \end{itemize}
\end{remark}

{\em In these lecture notes, we shall only consider a subclass of {\em quasi-split} $\imath$quantum groups, which can be formulated without use of braid group actions.} For example, the $\imath$quantum groups arising from $\imath$Schur duality and KL theory of (super) classical type in later sections will be quasi-split and type AIII.

A distinguished feature of the $\imath$quantum groups is that they depend on certain parameters $\bvs, \kappa$ (besides the quantum parameter $q$). Let us fix a set of representatives $\Ii$ for the regular $\tau$-orbits in $\I$.

\begin{definition}  [\cite{Let99, Ko14}]
  \label{def:Ui}
The {\em quasi-split $\imath${}quantum group}, denoted by $\Ui_{\vs,\kappa}$ or $\Ui$, is the $\Q(q)$-subalgebra of $\U$ generated by
\begin{align}
B_i :=E_{i}  &+ \vs_i F_{\tau i} K^{-1}_i + \kappa_i K^{-1}_{i}  \; (i \in \I), \quad k_i :=K_i K_{\tau i}^{-1}  \; (i \in \I).
  \label{eq:def:ff}
\end{align}
(Note that $E, F$ in \eqref{eq:def:ff} would be switched if we used Lusztig's comuliplication.)
Note $k_{\tau i} = k_i^{-1}$, for all $i\in \I$, and $k_{\tau i}=1$ if $\tau i=i$; hence it suffices to use $k_i$ for $i\in \Ii$ as generators of $\Ui$.
Here the parameters
\begin{equation}
  \label{parameters}
  \bvs=(\vs_i)_{i\in \I}\in (\Q(q)^\times)^\I,\qquad
  \kappa=(\kappa_i)_{i\in \I}\in\Q(q)^\I
\end{equation}
satisfy Conditions \eqref{kappa}--\eqref{vs=} below:
\begin{align}
 \label{kappa}
\kappa_i &=0, \quad \text{ unless } \tau i =i \text{ and } c_{ki} \in 2\Z \; (\mbox{for all $k = \tau k$});
\\
\vs_{i} & =\vs_{{\tau i}}, \quad \text{ if }    c_{i,\tau i} =0.
\label{vs=}
\end{align}
Here $c_{ki}$ is the $(k,i)$-th entry of the Cartan matric $C=(c_{ij})_{i,j\in\I}$.
\end{definition}

$\triangleright$ The conditions \eqref{kappa}--\eqref{vs=} on the parameters ensure that $\Ui$ has the expected size. Additional constraints on parameters will be needed for quasi K-matrix and $\imath$-canonical basis.

$\triangleright$ The quasi-split QSP $(\U, \Ui)$ or $\Ui$ in Definition~\ref{def:Ui} correspond to Satake diagrams without black nodes.

$\triangleright$ We call a quasi-split $\imath$quantum group $\Ui$  {\em split} if in addition $\tau=\mathrm{id}$. Note that a split $\imath$quantum group $\Ui$ is generated by $B_i \, \,(i \in \I)$.

$\triangleright$ It often occurs that $\kappa_i=0$ thanks to the condition \eqref{kappa}.
The algebras $\Ui_{\vs,\kappa}$ for various $\kappa$ are canonically isomorphic, and so it often suffices to consider
$\Ui_{\vs,0}$ on the algebra level; cf. \cite{Let02}.

$\triangleright$ The universal $\imath$quantum group $\tUi$ \cite{LuW19a} arising from the $\imath$Hall algebra construction is a subalgebra of the Drinfeld double $\tU$ of the quantum group $\U$. The algebra $\tUi$ contains various central elements, and then the $\imath${}quantum groups $\Ui_{\vs,0}$ are obtained from $\tUi$ by central reductions. This provides a conceptual explanation of the parameters $\bvs$ and a right setting for braid group symmetries.

\begin{remark}
\label{rem:iSerre}
The $\imath${}quantum groups $\Ui_{\vs,\kappa}$ of finite type admits a Serre type presentation \cite{Let99, Let02} (also cf. \cite{Ko14}).

A Serre presentation for {\em quasi-split} $\imath$quantum groups $\Ui_{\vs,\kappa}$ is much simpler; cf. for example, \cite[Theorem 3.1]{CLW18}. Actually the formulation therein is valid for {\em arbitrary Kac-Moody} type, where the $\imath$divided powers (introduced in \cite{BW18a, BeW18}) play a key role. We will not recall the Serre presentation in this generality, but we will make them explicit for $\Ui$ of type AIII in \S\ref{sec:AIII} below.
\end{remark}

\subsection{Quasi K-matrix}
 \label{subsec:K}

 In this section, we shall impose stronger conditions on the parameters $(\bvs, \kappa)$ besides \eqref{kappa}--\eqref{vs=}:
\begin{align}
 \vs_i, \kappa_i \in \A, & \quad \overline{\kappa_i} = \kappa_i;  \label{kappa2}
 \\
 \overline{\vs_iq_i} = \vs_iq_i, & \quad \text{ if } \tau i = i; 
 \\
 \overline{\vs_i} =\vs_i = \vs_{\tau i}, & \quad \text{ if } \tau i \neq i \text{ and } c_{i,\tau i} =0;
 \\
\vs_{{\tau i}} = q_i^{-c_{i,\tau i}} \overline{\vs_i}, & \quad \text{ if }    \tau i \neq i, c_{i,\tau i} \neq 0.   \label{vs=bar}
\end{align}
\noindent As we only consider $\Ui$ of finite type, \eqref{vs=bar} greatly simplifies as it occurs only in type AIII, where $c_{i,\tau i} =-1$; see \eqref{eq:Bi1} in \S\ref{subsec:odd} below.

Recall the bar involution $\psi: \U \rightarrow \U$ from Lemma~\ref{lem:inv}(3). As we restrict ourselves to quasi-split finite type cases, the following lemma easily follows by a direct inspection of the Serre presentation of $\Ui$; cf. Remark~\ref{rem:iSerre}. This was basically known to Bao and Wang \cite{BW18a}, who expected it to hold in general. The bar involution for $\imath$quantum groups of Kac-Moody type (without quasi-split condition) is then established in \cite{BK15} under some mild conditions on the Cartan matrices, where the conditions on parameters are determined precisely.

\begin{lemma}
  \label{lem:psi}
Assume that $(\bvs, \kappa)$ satisfy the conditions  \eqref{kappa}--\eqref{vs=} and \eqref{kappa2}--\eqref{vs=bar}. There exists a bar involution $\psi_\imath$ on the quasi-split $\imath$quantum group $\Ui$ such that $\psi_\imath(B_i)=B_i, \psi_\imath(k_j)=k_j^{-1}$ for all $i\in\I, j\in\Ii$.
\end{lemma}

Denote by $\imath: \Ui\rightarrow \U$ the embedding. Observe that the diagram
$$
\xymatrix{
  \Ui \ar[d]_{\psi_\imath} \ar[r]^{\imath}
                & \mathbf{U} \ar[d]^{\psi}  \\
  \Ui  \ar[r]^{\imath}
                & \mathbf{U}}$$
  is not commutative because $B_i=E_i+qF_{\tau i}K_i^{-1}\neq\psi(B_i)$.
  In other words, $$\overline{\imath}:=\psi\circ\imath\circ\psi_{\imath}\neq\imath.$$ Note the analogy with the fact that $\overline{\Delta} \neq \Delta$ on $\U$, i.e., the diagram
  $$
  \xymatrix{
  \mathbf{U} \ar[d]_{\psi} \ar[r]^{\Delta}
                & \mathbf{U}\otimes\mathbf{U} \ar[d]^{\psi\otimes\psi}  \\
  \mathbf{U}  \ar[r]^{\Delta}
                & \mathbf{U}\otimes\mathbf{U}}
 $$
is not commutative. Recall the quasi R-matrix $\Theta$ is used to intertwine $\overline{\Delta}$ and $\Delta$, i.e., $\Delta(\mu)\circ\Theta=\Theta\circ\overline{\Delta}(u)$, for all $u\in\mathbf{U}$; see Theorem~\ref{thm:quasi}.

  \begin{theorem}  \cite{BW18a} (also see \cite{BK19})
  \label{thm:quasi2}
Assume that $(\U, \Ui)$ is quasi-split. $(\bvs, \kappa)$ satisfy the conditions  \eqref{kappa}--\eqref{vs=} and \eqref{kappa2}--\eqref{vs=bar}. There exists a unique $\Upsilon=\sum_{\mu\in\mathbb{N}\I}\Upsilon_{\mu}$, where
    $\Upsilon_\mu\in\mathbf{U}_{-\mu}^{-}$ and $\Upsilon_0=1$, such that
    \[
    \imath(u)\circ\Upsilon=\Upsilon\circ\overline{\imath}(u), \quad \mbox{for all $u\in\Ui$}.
    \]
\end{theorem}
\noindent (If we have used the comultiplication in \cite{Lus93} and formulate the $\imath$quantum groups accordingly as in \cite{Ko14}, then $\Upsilon$ would lie in a completion of $\U^+$; cf. \cite{BK19, BW18b}.)

\begin{proof}
We refer to \cite[\S2.4]{BW18a} and \cite[\S6]{BK19} for details.
\end{proof}

The uniqueness of $\Upsilon$ implies readily that $\overline{\Upsilon}^{-1} =\Upsilon$, which can be reformulated as follows.

\begin{corollary}  \label{cor:inv2}
The quasi K-matrix $\Upsilon$ satisfies
  $\overline{\Upsilon}=\Upsilon^{-1}$.
\end{corollary}

Let  $(M,B)$ be a based $\mathbf{U}$-module. In particular, $M$ is involutive (i.e. $\psi(ux)=\psi(u)\psi(x)$, for all $u\in\mathbf{U}, x\in M$) as a $\mathbf{U}$-module. Define an anti-linear map \cite{BW18a}
\begin{align}  \label{eq:ibar}
\psi_\imath:=\Upsilon\circ\psi:M\rightarrow M.
\end{align}

\begin{lemma}
  \label{lem:ibar}
Let  $(M,B)$ be a based $\mathbf{U}$-module. Then
  \begin{itemize}
    \item[(1)] the map $\psi_\imath$ is a bar involution on $M$, i.e., $\psi_\imath^2=\text{id}$;
    \item[(2)] as a $\Ui$-module, $M$ is $\imath$-involutive, that is, $\psi_\imath(ux)=\psi_\imath(u)\psi_\imath(x)$, for all $u\in\Ui, x\in M$.
  \end{itemize}
\end{lemma}
\begin{proof}
This follows by direct computations by Theorem~\ref{thm:quasi2} and Corollary~\ref{cor:inv2}: $$\psi_\imath^2(x)=\Upsilon\psi(\Upsilon\psi(x))=
        \Upsilon\overline{\Upsilon}(\psi^2(x))=\Upsilon\overline{\Upsilon}(x)=x;$$ furthermore,   $$\psi_\imath(ux)=\Upsilon\psi(\imath(u)x)=\Upsilon\psi(\imath(u))\psi(x)
        =\imath(\psi_\imath(u))\Upsilon\psi(x)=\psi_\imath(u)\psi_\imath(x).$$
\end{proof}

Recall the $\A$-form $\mathbf{U}_{\A}^{-}$ of $\mathbf{U}^{-}$.

\begin{theorem}  \cite{BW18b, BW18c}
\label{thm:ZK}
  $\Upsilon\in\widehat{\mathbf{U}_{\A}^{-}}$, i.e., $\Upsilon_{\mu}\in\mathbf{U}_{\A}^{-}$ for any $\mu\in\mathbb{N}\I$.
\end{theorem}

\begin{proof}
We refer to \cite[Theorem 6.9]{BW18c} for details.
\end{proof}
\subsection{K-matrix}
 \label{subsec:KK}

Recall from Theorem~\ref{thm:R}(3) that some suitable twisting turns the quasi R-matrix $\Theta$ into the R-matrix $R:=\Theta\circ\widetilde{f}\circ P$, which is a $\U$-module isomorphism of the tensor product of $\U$-modules.

With a suitable twisting, we also turn the quasi K-matrix $\Upsilon$ into  a $\Ui$-module isomorphism $\mathcal T$ on a $\U$-module; for details we refer to \cite[Theorem~2.18]{BW18a} \cite[Theorem 7.5]{BK19} (also see \cite[Theorem 4.18, Remark 4.19]{BW18b}). This in particular induces a $\Ui$-module isomorphism:
\begin{align}   \label{eq:K}
\mathcal T: {}^\omega L(\la) \longrightarrow L(\la^\tau), \quad v^-_{-\la} \mapsto v^+_{\la^\tau}.
\end{align}
where we write $\la^\tau =\tau (\la)$. For us,  $\mathcal T$ in \eqref{eq:K} is all we need below in \eqref{xy:pi}.

Nowadays $\mathcal T$ has come to be known as a K-matrix, as it is shown \cite{BK19} that it provides a solution to the reflection equation, an $\imath$-analogue of Yang-Baxter equation.

\subsection{$\imath$Canonical basis on based $\U$-modules}
 \label{subsec:iCB}

A based $\U$-module is an involutive $\Ui$-module, by Lemma~\ref{lem:ibar}.

\begin{theorem}  \cite{BW18a, BW18b}
  Let $(M,B)$ be a based $\mathbf{U}$-module. Then there exists a unique basis $B^{\imath}=\{b^\imath~|~b\in B\}$ for $M$ such that
  $\psi_\imath(b^\imath)=b^{\imath}$ and $b^{\imath}\in b+\sum_{b'\in B,b'<b}q\Z[q]b'$, where $b'<b$ means $\mathrm{wt}(b)-\mathrm{wt}(b')\in\mathbb{N\I}$ and $b' \neq b$.

\noindent (The basis $B^\imath$ is called {\em $\imath$-canonical basis}, and $(M, B^\imath)$ a based $\Ui$-module.)
\end{theorem}

\begin{proof}
Note that $B$ is a basis for the $\A$-form $M_\A$ of $M$.
By definition \eqref{eq:ibar} of $\psi_\imath$ on $M$ and Theorem~\ref{thm:ZK}, we have, for $b\in B$,
\begin{align}  \label{eq:12i}
\psi_\imath (b) \in b +\sum_{b'\in B,b'<b}\A b'.
\end{align}
The theorem now follows from \eqref{eq:12i}, by exactly the same argument as in the proof of Theorem~\ref{thm:KL} for KL basis starting from the equation \eqref{eq:9} (which is the Hecke counterpart of \eqref{eq:12i}).
\end{proof}

\begin{corollary}  \label{cor:iCBtensor}
  Let $\lambda,\lambda^1,\ldots,\lambda^k\in X^+$. Then $L(\lambda)$ and $L(\lambda^1)\otimes\cdots\otimes L(\lambda^k)$ admit  $\imath$-canonical bases.
\end{corollary}

\subsection{$\imath$Canonical basis on modified $\imath$quantum groups}
 \label{subsec:iCB}

The quasi-split assumption on $\imath$quantum groups $\Ui$ in these notes greatly simplifies the constructions below, and we shall follow closely \cite{BW18a} instead of the more general constructions in \cite{BW18b}.

 Recall from \S3.1 X denotes the weight lattice. Define the $\imath$-weight lattice
\begin{equation}
X^\imath =X/\{ \la + \la^\tau~|~\la \in X\}, \qquad Y^\imath=\{\nu-\nu^\tau~|~\nu\in\mathbb{N}\I\},
\end{equation}
where we have denoted $\la^\tau =\tau(\la)$.
We shall denote the projection $X \rightarrow  X^\imath$, $\la \mapsto \overline \la$.

We define the modified version (i.e., idempotented version) $\dot\U^\imath $ of the $\imath$quantum groups similar to \cite[IV]{Lus93}; cf. \cite{BW18b, BW18c, Wat21}.

We define an $X_\imath$-grading on $\Ui$ by assigning $B_i$,  $k_i$ degree $-\overline{\alpha_i}$, $0$, respectively; cf. Definition~\ref{def:Ui}. We denote by $\U^\imath(\zeta)$ the homogeneous subspace of $\Ui$ of degree $\zeta \in X_\imath$. Note that $\U^\imath(\zeta)$ is non-trivial if and only if $\zeta \in \overline{\Z\I} \subset X_\imath$.

For any $\lambda', \lambda'' \in X_\imath$ with $\zeta = \lambda '' -\lambda'$, we define
\[
_{\lambda'}{\U}^\imath_{\lambda''} = \Ui(\zeta) \Big/ \Big( \sum_{\mu \in Y^\imath} (K_\mu - q^{( \mu, \lambda' )}) \Ui(\zeta) +  \sum_{\mu \in Y^\imath}  \Ui(\zeta) (K_\mu - q^{( \mu, \lambda'' )}) \Big).
\]
We denote by $\pi_{\lambda', \lambda''}: \Ui \rightarrow {}_{\lambda'}{\U}^\imath_{\lambda''}$ the quotient map, where $\pi_{\lambda', \lambda''}(x) =0$ if $x \not \in \Ui(\lambda'' - \lambda')$. We write ${\bf 1}_{\lambda'} = \pi_{\lambda', \lambda'}(1)$ for the orthogonal idempotents.

Following \cite[IV]{Lus93}, we then define an associative $\Q(q)$-algebra structure (without unit) on
\[
	\dot\U^\imath = \bigoplus_{\lambda', \lambda'' \in X_\imath} {}_{\lambda'}{\U}^\imath_{\lambda''}.
\]

For $\la \in X^+$, recall ${}^\omega L(\la)$ denotes the lowest weight $\U$-module of lowest weight $-\la$ with lowest weight vector $v^-_{-\la}$. One observes that ${}^\omega L(\la)$ is a cyclic module over $\Ui$ generated by $v^-_{-\la}$.

Given $\zeta \in X^\imath$, note that $\overline{\nu +\nu^\tau} =0 \in X^\imath$ for $\nu \in X$, so we can choose and fix $\la  \in X^+$ such that $-\overline \la =\zeta$. For each $\nu \in X^+$, ${}^\omega L(\nu +\nu^\tau +\la)$ admits $\imath$-canonical basis by Corollary~\ref{cor:iCBtensor}. Sending $v^-_{-\la-\nu-\nu^\tau}$ to $v^-_{-\la}$ (both of $\imath$-weight $\zeta$) defines a surjective $\Ui$-module homomorphism (see \cite[Proposition~ 4.4]{BW18a})
\[
\pi_{\la, \nu}: {}^\omega L(\nu +\nu^\tau +\la) \longrightarrow {}^\omega L(\la).
\]
More explicitly, the map $\pi_{\la, \nu}$ is a composition of several $\U$-module maps from \cite[\S25.1.1]{Lus93} and an $\Ui$-module isomorphism from the K-matrix $\mathcal T$ \eqref{eq:K}:
\begin{align} \label{xy:pi}
  \xymatrix{
  {}^\omega L(\nu +\nu^\tau +\la) \ar[d]_{\pi_{\la, \nu}} \ar[r]
                & {}^\omega L(\nu) \otimes {}^\omega L(\nu^\tau) \otimes {}^\omega L(\la) \ar[d]^{\mathcal T \otimes 1 \otimes 1}  \\
  {}^\omega L(\la)
                &  L(\nu^\tau) \otimes  {}^\omega L(\nu^\tau) \otimes {}^\omega L(\la)  \ar[l]_{\delta_{\nu^\tau} \otimes 1\qquad\quad}
                }
\end{align}

A basic property  (cf. \cite[Proposition 6.16]{BW18b}) is that $\pi_{\la, \nu}$ is {\em asymptotically compatible} with $\imath$-canonical bases (while being weaker than a homomorphism of based modules, it suffices for the purpose). Consequently, the projective system of $\Ui$-modules $\big\{ {}^\omega L(\la+\nu+\nu^\tau) \big\}_{\nu \in X^+}$ with asymptotically compatible $\imath$-canonical bases leads to an $\imath$-canonical basis on its inverse limit, and the inverse limit can be naturally identified with $\dot \U^\imath {\bf 1}_\zeta$.
In this way, we obtain an $\imath$-canonical basis on $\dot\U^{\imath}$. We refer to \cite[\S6]{BW18b} for details.

The $\imath$-canonical basis for $\Ui_q(\mf{sl}_2)$, known as $\imath$divided powers, admits explicit closed formulas \cite{BeW18}.

\section{$\imath$Schur duality}
  \label{sec:iSchur}

In this section, we shall formulate the $\imath$Schur duality between an $\imath$quantum group and a Hecke algebra of type B. The type B/D Kazhdan-Lusztig basis and $\imath$-canonical basis (with suitable parameters) on a tensor space are shown to coincide.

\subsection{Hecke algebras $\mathbb H$ with unequal parameters}

Let $$W=W_{B_m}=\langle s_i~|~0\leq i \le m-1 \rangle$$ be the Weyl group of type $B_m$, whose Dynkin/Coxter diagram is
\begin{figure}[ht!]
\caption{Dynkin/Coxter diagram of type $B_m$}
\begin{tikzpicture}[start chain]
\dnode{$0$}
\dnodenj{$1$}
\dydots
\dnode{$m-2$}
\dnode{$m-1$}
\path (chain-1) -- node[anchor=mid] {\(=\joinrel=\joinrel=\)} (chain-2);
\end{tikzpicture}
\end{figure}

The symmetric group $\mathfrak{S}_m$ is the subgroup of $W_{B_m}$ generated by $s_1,\ldots,s_{m-1}$.

The type B Hecke algebra allows 2 generic parameters $q, p$. For the sake of simplicity (keeping all constructions within the field $\Q(q)$), we shall work with a single parameter by assuming
\[
p\in q^\Z
\]
(or more generally, specifying any integral weight function in the sense of \cite{Lus03}). By definition, the Hecke algebra $\mathbb{H}=\mathbb{H}_{p,q}(B_m)$ (of type $B_m$) is a $\mathbb{Q}(q)$-algebra generated by $H_0,H_1,\ldots,H_{m-1}$ subject to the following relations:
\begin{itemize}
  \item quadratic relation:
  \begin{align*}
    (H_0+p)(H_0-p^{-1})=0,\quad (H_i+q)(H_i-q^{-1})=0,\ \mbox{for $1\leq i<m$};
  \end{align*}
  \item braid relation:
 $$\begin{array}{ll}
    H_0H_1H_0H_1=H_1H_0H_1H_0, &\\
    H_iH_{i+1}H_i=H_{i+1}H_iH_{i+1}, & \mbox{for $\leq i\leq d-2$},\\
    H_iH_j =H_jH_i, &\mbox{for $|i-j|>1$}.
  \end{array}$$
\end{itemize}
The $\Q(q)$-algebra $\mathbb H$ admits a bar involution $\overline{\phantom{x}}$: $q\mapsto q^{-1}$, $H_i \mapsto H_i^{-1}$, for each $i$. The type $A$ Hecke algebra $\H_q(\mathfrak{S}_m)$ is naturally a subalgebra of $\mathbb{H}$ generated by $H_1,\ldots,H_{m-1}$.

\begin{remark}
 \label{rem:Bpq}
In the same way as KL basis for equal parameter Hecke algebras was defined, we can define a canonical basis for $\mathbb{H}$ with unequal parameters; cf. \cite{Lus03}.
\end{remark}

\begin{remark}
 \label{rem:BD}
There are two important special cases of $\mathbb{H}$:
\begin{enumerate}
\item[(1)]
For $p=q$, $\mathbb H$ becomes the (equal parameter) Hecke algebra of type $B_m$, i.e.,  $\mathbb{H}_{q,q}(B_m) =\H_q(B_m)$.

\item[(2)]
For $p=1$, $H_0 =s_0 \in \mathbb{H}_{1,q}(B_m)$ satisfies $s_0^2=1$.
The type D Hecke algebra  $\H_q(D_m) =\langle H_0^{\mf d}, H_1, \ldots,H_{m-1} \rangle$ is naturally a subalgebra of $\mathbb{H}_{1,q}(B_m)$ via the embedding $H_0^{\mf d} \mapsto s_0 H_1 s_0, H_i \mapsto H_i$, for $1\le i <m$. A fundamental fact is that the type D KL basis for  $\H_q(D_m)$ is part of the canonical basis for $\mathbb{H}_{1,q}(B_m)$; this follows from the uniqueness of KL/canonical basis.
\end{enumerate}
\end{remark}

\subsection{Action of Hecke algebra $\mathbb{H}$ on $\mathbb{V}^{\otimes m}$}

As we shall consider the actions of the type B Weyl group, Hecke algebra and related quantum algebras, it is natural and convenient to switch the indexing set for the basis $\{v_i\}$ for the $N$-dimensional $\Q(q)$-space $\V$ to the integral interval $[{\scriptstyle \frac{1-N}{2}..\frac{N-1}{2}}]$ in the sense of \eqref{eq:interval}.
For example, $\mathbb{V}$ is spanned by $\{v_{-\frac{1}{2}},v_{\frac{1}{2}} \}$ if $N=2$, and by $\{v_{-1},v_{0},v_1\}$ if $N=3$.

As in \S\ref{sub:Hecke}, the $m$-fold tensor space $\mathbb{V}^{\otimes m}$ has a basis
\[
\left\{M_f=v_{f(1)}\otimes \cdots \otimes v_{f(m)}~\middle|~f \in \big[{\textstyle \frac{1-N}{2}..\frac{N-1}{2} } \big]^m \right\}.
\]
Set $fs_0=(-f(1),f(2),\ldots,f(m))$. We extend the action on  $\mathbb{V}^{\otimes m}$ of the type A Hecke algebra in Lemma~\ref{lem:HtypeA} (i.e., $H_i$ with $i\neq 0$) to type B as follows:
\begin{equation} \label{eq:HBaction}
  M_fH_i=\left\{
  \begin{array}{ll}
    M_{fs_i}, & \mbox{if $i\neq0$ and $f(i)<f(i+1)$};\\
    M_{fs_i}+(q^{-1}-q)M_f, & \mbox{if $i\neq0$ and $f(i)>f(i+1)$};\\
    q^{-1}M_f, & \mbox{if $i\neq0$ and $f(i)=f(i+1)$};\\
    M_{fs_0}, & \mbox{if $i=0$ and $f(1)>0$};\\
     M_{fs_0}+(p^{-1}-p)M_f, & \mbox{if $i=0$ and $f(1)<0$};\\
      p^{-1}M_{f}, & \mbox{if $i=0$ and $f(1)=0$}.
  \end{array}
  \right.
\end{equation}

The following lemma, extending the equal parameter case \cite{BW18a},  follows by a direct computation for $m=2$.

\begin{lemma} \cite{BWW18}
  The formula \eqref{eq:HBaction} defines an action of $\mathbb{H}$ on $\mathbb{V}^{\otimes m}$.
\end{lemma}

\subsection{$\imath$Quantum groups of type AIII}
  \label{sec:AIII}

The quantum symmetric pair $(\U,\Ui)$ used throughout this paper is quasi-split and type AIII, that is, $(\U,\Ui)\xrightarrow{q\mapsto1} \big(U(\mathfrak{sl}_N),U(\mathfrak{sl}_N^\theta) \big)$, where $\theta$ rotates $N\times N$ matrices by $180^\circ$, and $\tau$ is the nontrivial diagram involution. Note that $\mathfrak{sl}_N^\theta \cong \mathfrak{sl}_{\frac{N}2} \oplus \mathfrak{gl}_{\frac{N}2}$ for $N$ even, and $\mathfrak{sl}_N^\theta \cong \mathfrak{sl}_{\frac{N+1}2} \oplus \mathfrak{gl}_{\frac{N-1}2}$ for $N$ odd.

\subsubsection{$\Ui_q(\mathfrak{sl}_N)$, for $N$ even}

Now we consider the case for $N$ even, i.e.,
\[
N=2r+2, \qquad \text{ for } r\ge 0.
\]
Set \begin{align*}
&\I =\I_{2r+1} :=\{ -r, \ldots, -1, 0, 1, \ldots, r\},
\\
&\Ii :=\I \cap \Z_{>0} =\{1, \ldots, r\}.
\end{align*}

\begin{figure}[ht!]
\caption{Satake diagram of type AIII: $A_{2r+1}$ with involution $\tau \neq \mathrm{id}$}
   \label{figure:ii}
\begin{tikzpicture}
\matrix [column sep={0.6cm}, row sep={0.5 cm,between origins}, nodes={draw = none,  inner sep = 3pt}]
{
	&\node(U1) [draw, circle, fill=white, scale=0.6, label = $\scriptstyle -r$] {};
	&\node(U2) {$\cdots$};
	&\node(U3)[draw, circle, fill=white, scale=0.6, label =$\scriptstyle -1$] {};
\\
	&&&&
	\node(R)[draw, circle, fill=white, scale=0.6, label =$\scriptstyle 0$] {};
\\
	&\node(L1) [draw, circle, fill=white, scale=0.6, label =below:$\scriptstyle r$] {};
	&\node(L2) {$\cdots$};
	&\node(L3)[draw, circle, fill=white, scale=0.6, label =below:$\scriptstyle 1$] {};
\\
};
\begin{scope}
\draw (U1) -- node  {} (U2);
\draw (U2) -- node  {} (U3);
\draw (U3) -- node  {} (R);
\draw (L1) -- node  {} (L2);
\draw (L2) -- node  {} (L3);
\draw (L3) -- node  {} (R);
\draw (R) edge [color = blue,loop right, looseness=40, <->, shorten >=4pt, shorten <=4pt] node {} (R);
\draw (L1) edge [color = blue,<->, bend right, shorten >=4pt, shorten <=4pt] node  {} (U1);
\draw (L3) edge [color = blue,<->, bend left, shorten >=4pt, shorten <=4pt] node  {} (U3);
\end{scope}
\end{tikzpicture}
\end{figure}

The generators (or the simple roots) for $\U=\U_q(\mathfrak{sl}_{2r+2})$ are indexed by $\I$.
Define $\Ui = \Ui_q(\mathfrak{sl}_{2r+2})$  to be the $\mathbb{Q}(q)$-subalgebra of $\U$ generated by
\begin{align}
 k_i&=K_iK_{-i}^{-1} \; (i\in\Ii);
 \notag \\
 B_i&=E_i+F_{-i}K_i^{-1} \; (0\neq i\in\I),\quad B_0=E_0+qF_0K_0^{-1}+ \frac{p-p^{-1}}{q-q^{-1}} K_0^{-1}.
 \label{eq:Bi}
\end{align}
One checks that the $\mathbb{Q}(q)$-algebra $\Ui$ is a right coideal subalgebra of $\U$, i.e.,
$\Delta (\Ui) \subset \Ui\otimes \U.$

The pair $(\mathbf{U},\Ui)$ forms a QSP of type $AIII$, cf. \cite{Let99}. Note that the embedding $\Ui \rightarrow \U$ depends on the second parameter $p$ of type B Hecke algebra. However, as an abstract algebra, $\Ui$ admits a presentation below is independent of $p$.

\begin{proposition}   \cite{Let99}
 \label{prop:SP1}
The $\mathbb{Q}(q)$-algebra $\Ui$ admits a presentation as follows:
  \begin{itemize}
    \item generators: $B_i \, (i\in\I)$,\quad $k_j^{\pm1} \, (j\in\Ii)$;
    \item relations: $\left\{\begin{array}{l}
      k_ik_j=k_jk_i, \\
      k_iB_jk_i^{-1}=q^{c_{ij} -c_{-i,j}}B_j, \\
      {[B_i,B_{-i}]}=\frac{k_i-k_i^{-1}}{q-q^{-1}},  \qquad\qquad\mbox{ for $i \in \Ii$},\\
      {[B_i,B_j]}=0, \qquad \qquad\qquad\mbox{ unless $j=-i$ or $|i-j|=1$},\\
      B_i^2B_j-[2]B_iB_jB_i+B_jB_i^2=0, \quad\mbox{for $|i-j|=1$ and $i\neq0$},\\
      B_0^2B_j-[2]B_0B_jB_0+B_jB_0^2=B_j, \quad\mbox{for $j=\pm1$}.
    \end{array}\right.$
  \end{itemize}
\end{proposition}

\subsubsection{$\Ui_q(\mathfrak{sl}_N)$, for $N$ odd}
 \label{subsec:odd}
Now we consider the case for $N$ odd, i.e.,
\[
N=2r+1, \qquad \text{ for } r\ge 1.
\]
Introduce the indexing sets
\begin{align*}
&\I =\I_{2r}:= \left\{ {\textstyle \frac{1-2r}{2}, \ldots,  -\frac{1}{2},   \frac{1}{2},\frac{3}{2},\ldots,\frac{2r-1}{2} } \right \},
\\
&\Ij :=\I \cap \N= \left\{ {\textstyle \frac{1}{2},\frac{3}{2},\ldots, \frac{2r-1}{2} } \right \}.
\end{align*}
\begin{figure}[ht!]
\caption{Satake diagram of type AIII: $A_{2r}$ with involution $\tau \neq \mathrm{id}$}
   \label{figure:ij}
\begin{tikzpicture}
\matrix [column sep={0.6cm}, row sep={0.5 cm,between origins}, nodes={draw = none,  inner sep = 3pt}]
{
	&\node(U1) [draw, circle, fill=white, scale=0.6, label = $\scriptstyle \frac{1}{2}-r$] {};
	&\node(U3) {$\cdots$};
	&\node(U4)[draw, circle, fill=white, scale=0.6, label =$\scriptstyle -\frac{3}{2}$] {};
	&\node(U5)[draw, circle, fill=white, scale=0.6, label =$\scriptstyle -\frac{1}{2}$] {};
\\
	&&&&
\\
	&\node(L1) [draw, circle, fill=white, scale=0.6, label =below:$\scriptstyle r-\frac{1}{2}$] {};
	&\node(L3) {$\cdots$};
	&\node(L4)[draw, circle, fill=white, scale=0.6, label =below:$\scriptstyle \frac{3}{2}$] {};
	&\node(L5)[draw, circle, fill=white, scale=0.6, label =below:$\scriptstyle \frac{1}{2}$] {};
\\
};
\begin{scope}
\draw (U1) -- node  {} (U3);
\draw (U3) -- node  {} (U4);
\draw (U4) -- node  {} (U5);
\draw (U5) -- node  {} (L5);
\draw (L1) -- node  {} (L3);
\draw (L3) -- node  {} (L4);
\draw (L4) -- node  {} (L5);
\draw (L1) edge [color = blue,<->, bend right, shorten >=4pt, shorten <=4pt] node  {} (U1);
\draw (L4) edge [color = blue,<->, bend left, shorten >=4pt, shorten <=4pt] node  {} (U4);
\draw (L5) edge [color = blue,<->, bend left, shorten >=4pt, shorten <=4pt] node  {} (U5);
\end{scope}
\end{tikzpicture}
\end{figure}

Recall $p \in q^\Z$.

Let $\Ui =\Ui_q(\mathfrak{sl}_{2r+1})$ be the $\mathbb{Q}(q)$-subalgebra of $\mathbf{U}=\U_q(\mathfrak{sl}_{2r+1})$ generated by
\begin{align}
 k_i =K_iK_{-i}^{-1}\, (i\in\Ij),  & \qquad
 B_i=E_i+F_{-i}K_i^{-1}\, (i\in\I \backslash \{{\textstyle \pm \frac12} \}),
  \notag \\
 B_{{\frac12}} =E_{{\frac12}} + p^{-1} F_{-{\frac12}} K_{{\frac12}}^{-1},
  & \qquad
 B_{- {\frac12}} =E_{- {\frac12}} + p K_{- {\frac12}}^{-1} F_{ {\frac12}}.
 \label{eq:Bi1}
 \end{align}
The pair $(\mathbf{U},\Ui)$ forms a QSP of type $AIII$, cf. \cite{Let99, BWW18}.

\begin{proposition}   \cite{Let99}
 \label{prop:SP2}
  The $\mathbb{Q}(q)$-algebra $\Ui$ admits a presentation as follows:
  \begin{itemize}
    \item generators: $B_i \, (i\in\I)$,\quad $k_j^{\pm1} \, (j\in\Ij)$;
    \item relations:
    \[
    \left\{\begin{array}{l}
      k_ik_j=k_jk_i, \\
      k_iB_jk_i^{-1}=q^{(\alpha_i-\alpha_{-i},\alpha_j)}B_j, \\
      {[B_i,B_{-i}]}=\frac{k_i-k_i^{-1}}{q-q^{-1}},   \quad\qquad\mbox{ for $i \in \Ij$},\\
      {[B_i,B_j]}=0, \qquad \qquad\qquad\mbox{unless $j=-i$ or $|i-j|=1$},\\
      B_i^2B_j-[2]B_iB_jB_i+B_jB_i^2=0, \quad\mbox{for $|i-j|=1$ and $\{i,j\}\neq\{\pm\frac{1}{2}\}$},\\
      B_{\frac{1}{2}}^2B_{-\frac{1}{2}}-[2]B_{\frac{1}{2}}B_{-\frac{1}{2}}B_{\frac{1}{2}}
      +B_{-\frac{1}{2}}B_{\frac{1}{2}}^2=-[2] B_{\frac{1}{2}} (pq k_{\frac{1}{2}}+ p^{-1}q^{-1}k^{-1}_{\frac{1}{2}}),\\
      B_{-\frac{1}{2}}^2B_{\frac{1}{2}}-[2]B_{-\frac{1}{2}}B_{\frac{1}{2}}B_{-\frac{1}{2}}
      +B_{\frac{1}{2}}B_{-\frac{1}{2}}^2=-[2] (pq k_{\frac{1}{2}}+ p^{-1}q^{-1}k^{-1}_{\frac{1}{2}}) B_{-\frac{1}{2}}.
    \end{array}\right.
    \]
  \end{itemize}
\end{proposition}

\begin{remark}
 \label{rem:relations}
All defining relations in a Serre presentation for a general $\imath$quantum group $\Ui$ are local and arise from rank 1 or rank 2 relations. In particular, for (quasi-)split ADE type, a presentation for $\Ui$ can be formulated using the relevant (rank 1 or 2) relations listed in Propositions~\ref{prop:SP1}--\ref{prop:SP2}.
\end{remark}

\subsection{$\imath$Schur duality}

Now we consider $\Ui =\Ui_q(\mathfrak{sl}_N)$ for all $N\ge 2$ together. The action of $\U$ on $\V^{\otimes m}$ restricts to the subalgebra $\Ui$. On the other hand, The action of $\H_q(B_m)$ on $\V^{\otimes m}$ is given by the formula \eqref{eq:HBaction} (where the last case when $f(1)=0$ does not occur for $N$ odd).

The following is a type B generalization of the Jimbo-Schur duality, cf. Theorem~\ref{thm:Jimbo}. Recall $p\in q^\Z$, including the most fundamental cases when $p=q$ or $p=1$.

\begin{theorem}[$\imath$-Schur duality  \cite{BW18a, Bao17, BWW18}]
 \label{thm:iSchur}
 The actions of $\Ui$ and $\mathbb{H}_{p,q}(B_m)$ on $\V^{\otimes m}$
commute and satisfy the double centralizer property.
\end{theorem}

\begin{proof}
We shall only check the commutativity of the two actions (and then the double centralizer property follows by some deformation argument by specializing to $q\mapsto 1$.)

Since the actions of $\Ui$ and $\mathbb{H}_{p,q}(B_m)$ here are compatible with the commuting actions of $\U$ and $\mathbb{H}_1(\mathfrak S_d)$ in Jimbo-Schur duality,  it remains to verify the commutativity of the actions of $H_0$ and $\Ui$. This reduces further to the case when $m=1$, which is then checked directly.
\end{proof}

\begin{remark}
 \label{rem:K}
In the spirit of Jimbo's realization of type A Hecke algebra via R-matrix (Theorem~\ref{thm:Jimbo2}),
the action of $H_0$ can be realized as $\mathcal T^{-1} \otimes \text{id}^{\otimes m-1}$ via the K-matrix $\mathcal T$ in \eqref{eq:K} (note that $L(\la^\tau) \cong {}^\omega L(\la)$ as $\U$-modules); see \cite[Theorem 5.4]{BW18a}.
\end{remark}

\begin{remark}
The commuting action between $\Ui$ for $p=1$ and Hecke algebra of type $D_m$ on $\V^{\otimes m}$ was independently observed in \cite{ES18}.
\end{remark}

\subsection{$\imath$Canonical basis vs KL basis of classical type}

Recall the $\imath$-Schur duality $\Ui\curvearrowright \V^{\otimes m}\curvearrowleft \mathbb{H}_{p,q}(B_m)$; see Theorem~\ref{thm:iSchur}. A weight $f \in \big[{\textstyle \frac{1-N}{2}..\frac{N-1}{2} } \big]^m$ is called {\em $\imath$-anti-dominant} if $0\leq f(1)\leq\cdots\leq f(m)$.

\begin{lemma} \cite{BW18a, BWW18}
\label{lem:3barb}
The bar involution $\psi_\imath$ on $\V^{\otimes m}$ as a $\Ui$-module satisfies that
\[
    \psi_\imath(uxh)=\psi_\imath(u)\psi_\imath(x)\overline{h},
    \quad \text{ for } u\in\Ui, x\in\V^{\otimes m}, h\in\mathbb{H}_{p,q}(B_m).
\]
Moreover, such a bar involution is unique by requiring that $\psi (M_f)=M_f$ for $f$ being $\imath$-anti-dominant.
\end{lemma}

\begin{proof}
We refer to the proof of \cite[Theorem 5.8]{BW18a} for details.
\end{proof}

Let us take $p=q$ for now.
As a $\H_q(B_m)$-module,
\[
\V^{\otimes m} \cong\bigoplus_{f} C_{J(f)}\H_q(B_m), \quad
M_f\mapsto C_{J(f)},
\]
where $f$ runs over the set of all $\imath$-anti-dominant weights. Then $\V^{\otimes m}$ admits a type $B$ Kazhdan-Lusztig basis via the isomorphism above.

Similarly, for $p=1$, $\V^{\otimes m}$ admits a type $D$ Kazhdan-Lusztig basis by its decomposition as a direct sum of permutation modules over $\mathbb{H}_{1,q}(B_m)$; cf. Remark~\ref{rem:BD}.

Recall the embedding $\Ui \rightarrow \U$ depends on $p\in q^\Z$; cf. \eqref{eq:Bi}--\eqref{eq:Bi1}.

\begin{theorem} \cite{BW18a, Bao17}
 \label{thm:CBKLb}
 (1) The $\imath$-canonical basis (for $p=q$) and the type $B$ KL basis on $\V^{\otimes m}$ coincide;

 (2) The $\imath$-canonical basis (for $p=1$) and the type $D$ KL basis on $\V^{\otimes m}$ coincide.
\end{theorem}

\begin{proof}
Follows from Lemma~\ref{lem:3barb} and the uniqueness of such bar-invariant bases (which lie in the same $\Z[q]$-lattice).
\end{proof}

\begin{example}
  The natural representation $\V$ of $\U_q(\mathfrak{sl}_{2r+2})$ admits a canonical basis $\big\{v_i|~i \in   [{\scriptstyle -r-\frac{1}{2} .. r+\frac{1}{2}}] \big\}$. For $p=q$, the $\imath$-canonical basis for $\V$ is $\big\{v_i, v_{-i}+qv_i~|~i\in [{\scriptstyle \frac{1}{2}..r+\frac{1}{2} }] \big\}$.
\end{example}

\begin{remark}
The $\imath$-canonical basis (for any $p\in q^\Z$) and the corresponding unequal parameter KL basis on $\V^{\otimes m}$ coincide, cf.  \cite{BWW18}. The two distinguished cases as in Theorem~\ref{thm:CBKLb} are most fundamental.
\end{remark}

%

\section{Character formulas via canonical bases}
  \label{sec:KL}

In this section, we reformulate Kazhdan-Lusztig theory for category $\mathcal{O}$ of type A via  $q$-Schur duality and Lusztig's canonical basis on a tensor product of modules.
We also reformulate Kazhdan-Lusztig theory of classical type via  $\imath$Schur duality and $\imath$-canonical basis. Then we generalize the ($\imath$-) canonical basis character formulas to BGG category $\mathcal O$ for Lie superalgebras of classical types \cite{CLW15, BW18a, Bao17}.

References for the basics on the BGG category $\mathcal O$ and its super version are \cite{Hum08} and \cite{CW12}.

\subsection{Kazhdan-Lusztig theory} 

Let $\mathfrak{g}=\mf{n}^-\oplus\mf{h}\oplus\mf{n}$ be a semisimple (or reductive) Lie algebra over $\mathbb{C}$. The BGG category $\mathcal{O}$ consists of $\mathfrak{g}$-modules satisfying
\begin{itemize}
  \item[($\mathcal{O}1$)] $M$ is a finitely generated $U(\mathfrak{g})$-module;
  \item[($\mathcal{O}2$)] $M=\bigoplus_{\mu\in\mathfrak{h}^*}M_\mu$ is a weight module;
  \item[($\mathcal{O}3$)] $M$ is locally $\mathfrak{n}$-finite, that is, the subspace $U(\mathfrak{n})\cdot v$ is finite dimensional for any $v\in M$.
\end{itemize}
There is a duality functor $\vee:\mathcal{O}\rightarrow\mathcal{O}$ sending $M=\bigoplus_{\mu\in\mathfrak{h}^*}M_\mu$ to $M^\vee:=\bigoplus_{\mu\in\mathfrak{h}^*}M_\mu^*$.

Let $M(\lambda)$ be the Verma module with highest weight $\lambda$ and $L(\lambda)$ be its unique irreducible quotient. The projective module $P(\lambda)$ is the projective cover of $L(\lambda)$. The tilting module $T(\lambda)$ is the unique indecomposable module with highest weight $\lambda$ that has both Verma filtration and dual Verma filtration. It is known that the tilting module $T(\lambda)$ is self-dual with respect to $\vee$.

The BGG category has a block decomposition $$\mathcal{O}=\bigoplus_{[\lambda]\in X/W\bullet}\mathcal{O}_\lambda, $$ where $X/W\bullet$ is the set of $W$-orbits (under dot action $\sigma\bullet\lambda=\sigma(\lambda+\rho)-\rho$) in $X$.
We call $\mathcal{O}_0$ (which contains the trivial representation $L(0)$) the principal block. If $\lambda\in X$ is dot regular (i.e. $|W\bullet\lambda|=|W|$), then $\mathcal{O}_{\lambda}\cong\mathcal{O}_0$.

In the principal block $\mathcal{O}_0$, the simple objects $\{L(\sigma):=L(\sigma\bullet(-2\rho))~|~\sigma\in W\}$ are parameterized by $W$. There are similar notations $M(\sigma):=M(\sigma\bullet(-2\rho))$ and $T(\sigma):=T(\sigma\bullet(-2\rho))$.

Both $\mathcal{O}$ and $\mathcal{O}_0$ are abelian categories.
Let $[\mathcal{O}_0]$ be the Grothendieck group of $\mathcal{O}_0$. The standard, canonical, and dual canonical bases on the group algebra $\Z W$ are obtained from the corresponding bases on $\HA$, thanks to the base change $\Z W=\Z\otimes_{\A}\HA$, where $\A$ acts on $\Z$ with $q \mapsto 1$. Below is a modern formulation of the celebrated Kazhdan-Lusztig conjecture \cite{KL79}, where the tilting module statement follows by the Sorgel duality \cite{So98}.

\begin{theorem}  [Brylinski-Kashiwara \cite{BK81}, Beilinson-Bernstein \cite{BB81}]
  \label{thm:KLH}\
  \\
  The $\Z$-isomorphism $[\mathcal{O}_0]\rightarrow \Z W$, $[M(\sigma)]\mapsto H_\sigma \; (\sigma \in W)$, sends
  \begin{equation*}
    [L(\sigma)]\mapsto L_{\sigma};\quad
    [T(\sigma)]\mapsto C_\sigma, \qquad \text{for }\sigma \in W.
  \end{equation*}
\end{theorem}

\subsection{(Super) type A character formulas via canonical bases}

We specialize to type A, and consider the BGG category $\mathcal{O}^{m}$ for modules over Lie algebra $\mathfrak{g} =\mathfrak{gl}_{m}$ with weights in the weight lattice $X^m :=\bigoplus_{i=1}^m\Z\epsilon_i$. We can take the Weyl vector for $\mathfrak{g}$ to be
\[
\rho= m \epsilon_1+(m-1)\epsilon_2+\ldots+\epsilon_{m},
\]
which is the half sum of positive roots up to a shift by a multiple of $(\epsilon_1+\epsilon_2+\cdots+\epsilon_m)$.
Let $\V=\langle v_i~|~i\in\Z\rangle$ be the natural $\U_q(\mathfrak{sl}_{\infty})$-module. It has a basis $\{M_f~|~f\in\Z^m\}$. There exists a natural bijection $X^m  \stackrel{\cong}{\rightarrow} \Z^m$, which maps
\[
\lambda \mapsto f_\lambda :=(f_\lambda(1),\ldots,f_\lambda(m)), \quad \text{ where } \lambda+\rho =\sum_{i=1}^m f_\lambda(i) \epsilon_i.
\]

Below is a reformulation of Kazhdan-Lusztig theory of type $A$ via the canonical basis. Denote by $\V_\Z:=\Z\otimes_{\A}\V_{\A}$, a $\Z$-module.  The standard basis $\{M_f\}$, canonical basis $\{C_f\}$, and dual canonical basis $\{L_f\}$ on $\V_\Z^{\otimes m}$ are descended from the corresponding bases on $\V^{\otimes m}$.

\begin{theorem} [Type A KL reformulated]
\label{thm:KLa}
  The $\Z$-module isomorphism $[\mathcal{O}^m]\rightarrow \V_\Z^{\otimes m}$, $[M(\lambda)]\mapsto M_{f_\lambda}$, for $\lambda \in X$, sends
  \begin{equation*}
    [L(\lambda)]\mapsto L_{f_{\lambda}}, \qquad [T(\lambda)]\mapsto C_{f_{\lambda}},
    \quad \text{ for all }\lambda \in X.
  \end{equation*}
\end{theorem}

\begin{proof}
It is established \cite{So90} that the parabolic KL basis solve the irreducible character problem in a singular block of $\mathcal O$ -- this is a parabolic version of Theorem \ref{thm:KLH}. Now the type A KL reformulation follows by Theorems~\ref{thm:CBKLa} and \ref{thm:KLH}.
\end{proof}

One great advantage of the above reformulation of type A Kazhdan-Lusztig theory is that it allows a natural generalization to Lie superalgebras. A basic fact here  is that the Chevalley generators $E_i, F_i$ in $\U$ act on $[\mathcal{O}^m]$ by translation functors; this makes the above $\Z$-module isomorphism a $\U$-module isomorphism at $q=1$. This continues to make sense for a suitable formulation of super type A Kazhdan-Lusztig theory \cite{Br03}, which will be outlined below.

Now we take $\mathfrak{g}=\mathfrak{gl}(m|n)$ to be the general linear Lie superalgebra, cf. \cite{CW12}. Again, $\mathfrak{g}$ admits a standard triangular decomposition, and the BGG category $\mathcal O^{m|n}$ of $\mathfrak{g}$-modules (with weights in $X^{m|n}$) can be similarly defined, where
\[
X^{m|n}=\bigoplus_{i=1}^m\Z\epsilon_i\oplus\bigoplus_{j=1}^n \Z\delta_j.
\]
The category $\mathcal O^{m|n}$  admits Verma modules $M(\lambda)$, tilting modules $T(\lambda)$, and simple modules $L(\lambda)$, for all $\lambda \in X^{m|n}$.
The Weyl vector $\rho$ for $\mathfrak{g}=\mathfrak{gl}(m|n)$ can be taken to be
\[
\rho = m \epsilon_1+(m-1)\epsilon_2+\ldots+\epsilon_{m}  -\delta_1 -2\delta_2 -\ldots -n \delta_n.
\]

The Weyl group of $\mathfrak{g}$, which is $\mathfrak{S}_m\times\mathfrak{S}_n$, no longer controls the linkage principle in $\mathcal{O}^{m|n}$, for $m, n\geq 1$. One implication of this is that we can not expect to formulate a Kazhdan-Lusztig theory and solve the irreducible character problem for $\mathfrak{gl}(m|n)$ via Hecke algebra directly.

However, by Corollary~\ref{cor:CBtensor}, the $\U_q(\mathfrak{sl}_{\infty})$-module $\V^{\otimes m}\otimes \V^{*\otimes n}$ admits a canonical basis $\{C_f\}$ (and a dual canonical basis $\{L_f\}$ analogous to the dual KL basis in Proposition~\ref{prop:dualKL}). There is a natural bijection
\begin{align*}
  X^{m|n}
  &\xrightarrow[\mbox{\tiny $\rho$-shift}]{\cong}\Z^{m+n},\quad
  \lambda \mapsto f_\lambda  := \big(f_\lambda(1),\ldots,f_\lambda(m+n) \big),
\end{align*}
where we have denoted
\begin{align} \label{eq:shift}
\lambda+\rho =\sum_{i=1}^m f_\lambda(i) \epsilon_i + \sum_{j=1}^n f_\lambda(m+j) \delta_j.
\end{align}
This bijection provides us a $\Z$-module isomorphism
\begin{align}  \label{eq:Psi}
\Psi: [\mathcal{O}^{m|n}]\longrightarrow \V_\Z^{\otimes m} \otimes \V_{\Z}^{*\otimes n},
\qquad [M(\lambda)]\mapsto M_{f_\lambda}\; \; (\lambda \in X^{m|n}).
\end{align}

We have the following generalization of Theorem~\ref{thm:KLa}, which was conjectured by Brundan \cite{Br03}. We refer to \cite{CLW15} for details of a proof.

\begin{theorem} [Cheng-Lam-Wang \cite{CLW15}]
\label{thm:KLa2}
  The $\Z$-module isomorphism $\Psi: [\mathcal{O}^{m|n}]\rightarrow \V_\Z^{\otimes m} \otimes \V_{\Z}^{*\otimes n}$ in \eqref{eq:Psi} sends
  \begin{equation*}
[L(\lambda)]\mapsto L_{f_{\lambda}}, \quad [T(\lambda)]\mapsto C_{f_{\lambda}}, \quad \text{ where } \lambda \in X^{m|n}.
  \end{equation*}
\end{theorem}

\subsection{Classical type character formulas via $\imath$-canonical bases}

Another infinite family of Lie superalgebras is the ortho-symplectic superalgebras, including $\mathfrak{osp}(2m+1|2n)$ of type B and $\mathfrak{osp}(2m|2n)$ of type $D$ (the case when $m=1$ is also known as type C), cf. \cite{CW12}.

Let us treat the type B case: $\mathfrak{g}=\mathfrak{osp}(2m+1|2n)$ in detail. In this case, the Weyl vector $\rho$ has coefficients in ${\textstyle \frac{1}{2} } +\Z$. There exists a $\rho$-shift bijection for the set of integer weights
\begin{align}   \label{XZ}
X^{m|n}=\bigoplus_{i=1}^m\Z\epsilon_i\oplus\bigoplus_{j=1}^n \Z\delta_j
  \xrightarrow[\mbox{\tiny $\rho$-shift}]{\cong}({\textstyle \frac{1}{2} } +\Z)^{m+n}, \quad \lambda\mapsto f_\lambda,
  \end{align}
 where $f_\lambda$ is formally defined in the same way as in \eqref{eq:shift}.
Similarly, there exists a $\rho$-shift bijection for the set of half integer weights
 \begin{align}   \label{XZhalf}
  X^{m|n}_{{\scriptstyle \frac{1}{2}}}
  =\bigoplus_{i=1}^m({\textstyle \frac{1}{2} } +\Z)\epsilon_i\oplus
  \bigoplus_{j=1}^n ({\textstyle \frac{1}{2} } +\Z)\delta_j
  \xrightarrow[\mbox{\tiny $\rho$-shift}]{\cong}\Z^{m+n}, \quad \lambda\mapsto f_\lambda.
  \end{align}
  Again we can introduce the BGG category $\mathcal{O}^{m|n}_{\mf b}$ (respectively, $\mathcal{O}^{m|n}_{\mf b, \frac12}$), which contains the Verma modules $M(\lambda)$, tilting modules $T(\lambda)$ and simple modules $L(\lambda)$, for $\la \in X^{m|n}$ (respectively, $\la \in X^{m|n}_{{\scriptstyle \frac{1}{2}}}$). The Weyl group ($\cong B_m\times \mathfrak{S}_n$) does not control the linkage principle in the BGG categories for $\mathfrak{g}=\mathfrak{osp}(2m+1|2n)$.

We shall need the quantum symmetric pairs of type AIII, $(\U_q(\mathfrak{sl}_N), \Ui)$, from \S\ref{sec:AIII}, where we fix $p=q$ in \eqref{eq:Bi} and \eqref{eq:Bi1}; also recall the natural representation $\V$ with basis $\{v_i~|~i \in [\textstyle{\frac{1-N}{2}..\frac{N-1}{2}}]\}$, for $N$ even and odd, allowing $N =\infty$ (with parity!). Then we can identify the indexing set $[\textstyle{\frac{1-N}{2}..\frac{N-1}{2}}]$ with ${\textstyle \frac{1}{2} } +\Z$  for $N=\infty$ (even), and with $\Z$  for $N=\infty$ (odd).

  By Corollary~\ref{cor:iCBtensor}, the $\U_q(\mathfrak{sl}_{\infty})$-module $\V^{\otimes m}\otimes \V^{*\otimes n}$ (regarded as $\Ui$-module with $p=q$) admits an $\imath$-canonical basis, denoted by $\{C^\imath_f\}$, and a dual $\imath$-canonical basis, denoted by $\{L^\imath_f\}$, where $f \in ({\textstyle \frac{1}{2} } +\Z)^{m+n}$ or $\Z^{m+n}$, respectively.

  Define the following $\Z$-module isomorphisms
  \begin{align}
  \label{eq:P2}
  \begin{split}
  \Psi_{\mf b}: [\mathcal{O}^{m|n}_{\mf b}]\rightarrow \V_\Z^{\otimes m} \otimes \V_{\Z}^{*\otimes n}, & \quad
  [M(\lambda)]\mapsto M_{f_\lambda}\;\; (\la \in X^{m|n}),
  \\
  \Psi_{\mf b, \frac12}: [\mathcal{O}^{m|n}_{\mf b,\frac12}]\rightarrow \V_\Z^{\otimes m} \otimes \V_{\Z}^{*\otimes n}, & \quad
  [M(\lambda)]\mapsto M_{f_\lambda} \;\; (\la \in X^{m|n}_{{\scriptstyle \frac12}}).
  \end{split}
  \end{align}
  A basic fact here \cite[Proposition~11.9]{BW18a} is that the Chevalley generators $B_i$ in $\Ui$ act on $[\mathcal{O}^{m|n}_{\mf b}]$ and $[\mathcal{O}^{m|n}_{\mf b,\frac12}]$ by translation functors; thus the above $\Z$-module isomorphisms become $\Ui_\mathbb{Z}$-module isomorphisms at $q=1$. (A similar observation on translation functors and $B_i$ is valid for $p=1$ \cite{Bao17}, and it was made independently in \cite{ES18} in  the non-super setting, i.e.,  $m=0$.)

\begin{theorem} \cite{BW18a}
  \label{thm:KLb}
  The $\Z$-module isomorphism $  \Psi_{\mf b}$ (respectively, $\Psi_{\mf b, \frac12}$) in \eqref{eq:P2} sends
  \begin{equation*}
  [L(\lambda)]\mapsto L^\imath_{f_{\lambda}}, \quad [T(\lambda)]\mapsto C^\imath_{f_{\lambda}},
  \end{equation*}
 for $\la \in X^{m|n}$ (and respectively, $\la \in X^{m|n}_{{\scriptstyle \frac12}}$).
\end{theorem}

\begin{proof}
The strategy is quite similar to the proof of Theorem~\ref{thm:KLa} for super type A, once we have the (dual) $\imath$-canonical bases available.
\end{proof}

\begin{example}
  Take $n=0$ and $m=1$, so that $\mathfrak{g}=\mathfrak{so}_3\cong\mathfrak{sl}_2$. If the standard basis (which is the same as canonical basis) $\{v_i~|~i\in\Z\}$ for $\V$ is indexed by $\Z$, then $\V$ admits an $\imath$-canonical basis is $\{v_0, v_i, v_{-i}+q v_i~|~i\in\Z_{>0}\}$ and a dual $\imath$-canonical basis $\{v_0, v_i,v_{-i}-q^{-1}v_i~|~i\in\Z_{>0}\}$. Similar formulas hold when the standard basis of $\V$ is indexed by ${\textstyle \frac{1}{2} } +\Z$.
\end{example}

\begin{remark}
\begin{itemize}
\item[(1)]
Theorem~\ref{thm:KLb} for $m=0$ is a reformulation for the usual type B KL theory, thanks to Theorems~\ref{thm:CBKLb}(1) and \ref{thm:KLH}.
\item[(2)]
Theorem~\ref{thm:KLb} can be conceptually adapted to the super type D case, $\mathfrak{g}=\mathfrak{osp}(2m|2n)$, according to Bao \cite{Bao17}: one just needs to set the parameter $p=1$ (instead of $p=q$) in \eqref{eq:Bi} and \eqref{eq:Bi1}; compare Theorem~\ref{thm:CBKLb}(2).
\item[(3)]
The character formulas in Theorem~\ref{thm:KLa2} for super type A, Theorem~\ref{thm:KLb} for super type B, and its super type D analogue admit a natural generalization for a general parabolic BGG category of $\mathfrak{g}$-modules, where $ \V_\Z^{\otimes m} \otimes \V_{\Z}^{*\otimes n}$ is replaced by a tensor product of $q$-wedge spaces (which reflects the shape of a chosen Levi subalgebra of $\mathfrak{g}$). The formulation in the full generality for type B/D can be found in \cite{BWW20}.
\end{itemize}
\end{remark}

%



\section{A geometric setting for Schur and $\imath$Schur dualities}
 \label{sec:geom}

In this section, a geometric realization via flag varieties of $q$-Schur duality and $\imath$Schur duality as well as the corresponding ($\imath$-)canonical bases is provided.

\subsection{Iwahori's geometric realization of $\H$}

Let $\mathbb{F}_{\mathbf{q}}$ be a finite field of $\mathbf{q}$ elements of characteristic $> 3$. Let $\mathbf{G}$ be a connected algebraic group
defined over $\mathbb{F}_\mathbf{q}$. Assume $G=\mathbf{G}(\mathbb{F}_\mathbf{q})$ admits a split maximal torus and Borel subgroup, denoted by $T\subset B$. Let $W=N(T)/T=\langle s_i~|~i\in\I\rangle$ be the Weyl group of $G$.

\begin{lemma}
  \begin{itemize}
    \item[(1)] There are natural bijections:
     \[
     G/B\times_{G} G/B \quad\longleftrightarrow \quad B\setminus G/B \quad \longleftrightarrow \quad W.
     \]
  In particular, we have the Bruhat decompositions:
  \[
  G=\bigsqcup_{\sigma\in W}B\sigma B, \quad\quad
      G/B=\bigsqcup_{\sigma\in W}B\sigma B/B.
      \]
    \item[(2)] $B\sigma B\cdot Bs_iB=\left\{\begin{array}{ll}
      B\sigma s_i B, & \mbox{if $\ell(\sigma s_i)=\ell(\sigma)+1$};\\
      B\sigma s_iB\cup B\sigma B, & \mbox{otherwise}.
    \end{array}\right.$
  \end{itemize}
\end{lemma}

The Iwahori-Hecke algebra
\[
\mathcal{IH}:=\mathrm{End}_G(\mathrm{Ind}_B^G \C),
\]
where $\C$ is the trivial $B$-module, can be reformulated in several equivalent ways:
\[
\mathcal{IH} \cong
\mathrm{End}_G(\mathbb{C}[G/B])=\mathbb{C}[B\setminus G/B]=\mathrm{Fun}_G(G/B\times G/B) \cong e\mathbb{C}Ge,
\]
where $e=\frac{1}{|B|}\sum\limits_{b\in B}b$ is idempotent. There is a standard basis for $\mathcal{IH}$: $T_w=|B|^{-1}\sum\limits_{x\in BwB}x$, for $w\in W$. The algebra structure on $\mathcal{IH}$ corresponds to a convolution product on $\mathrm{Fun}_G(G/B\times G/B)$. Recall the (generic) Hecke algebra $\H =\H_q(W)$ from \S\ref{subsec:Hecke}, and we have a $\C$-algebra $\H_q(W)|_{q=-\mathbf{q}^{\frac{1}{2}}} =\C \otimes_{\A} \H_q(W)$, where $q\in \A$ acts on $\C$ by $-\mathbf{q}^{\frac{1}{2}}$.

\begin{theorem} [Iwahori \cite{Iw64}, cf. \cite{DDPW}]
There is an $\A$-algebra isomorphism
  \begin{align*}
    \mathcal{IH}\cong \H_q(W)|_{q=-\mathbf{q}^{\frac{1}{2}}}, \quad -\mathbf{q}^{-\frac{1}{2}}T_{s_i}\mapsto H_i.
  \end{align*}
\end{theorem}

\subsection{$q$-Schur algebras in geometry}

We shall realize the $q$-Schur algebras and $\imath$Schur algebra via flag varieties of classical type.

\subsubsection{Type $A$}

Let $G=\mathrm{GL}_m(\mathbf{q})$.
The complete flag variety (of type $A$) is
$$\mathscr{B}:=\{0=V_0\subset V_1\subset\cdots\subset V_m=\mathbb{F}_\mathbf{q}^m~|~\dim V_i=i, 0\leq i\leq m\}$$
and the $N$-step flag variety is
\[
\mathscr{F}=\mathscr{F}_{N,m}:= \Big\{0=V_0\subset V_1\subset\cdots\subset V_N=\mathbb{F}_\mathbf{q}^m \Big\}.
\]
It is decomposed into connected components (which are partial flag varieties) as
\[
\mathscr{F}= \bigsqcup_{\underline{m} =(m_1,\ldots,m_N)\models m} \mathscr{F}_{\underline{m}}.
\]
As varieties, $$\mathscr{B}\simeq G/B\quad \mbox{and}\quad \mathscr{F}\simeq \bigsqcup_{\underline{m}=(m_1,\ldots,m_N)\models m}G/P_{\underline{m}}.$$

Clearly $G$ acts on $\mathscr{F}$ and $\mathscr{B}$. Let $G$ act diagonally on $\mathscr{F}\times \mathscr{F}$, $\mathscr{F}\times \mathscr{B}$ and $\mathscr{B}\times\mathscr{B}$, respectively.
Denote by $\Theta$ the set of $G$-orbits in $\mathscr{F}\times\mathscr{F}$. There is a convolution product $\ast$ on the space of $G$-invariant functions $\mathrm{Fun}_G(\mathscr{F}\times\mathscr{F})$ defined as follows. For a triple $(\xi,\xi',\xi'')\in\Theta\times\Theta\times\Theta$, we fix an element $(f_1,f_2)$ in the $G$-orbit $\xi''$, and let $\kappa_{\xi,\xi',\xi'';\mathbf{q}}$ be the number of $f\in\mathscr{F}$ such that $(f_1,f)\in\xi$ and $(f,f_2)\in\xi'$. One can prove that there exists a polynomial $\kappa_{\xi,\xi',\xi''}\in\Z[q^2]$ such that $\kappa_{\xi,\xi',\xi'';\mathbf{q}}=\kappa_{\xi,\xi',\xi''}|_{q^{-1}=\mathbf{q}}$ for all prime powers $\mathbf{q}=p^r$ with $p \ge 3$. The convolution product $\ast$ on $\mathrm{Fun}_G(\mathscr{F}\times\mathscr{F})$ is defined by
$$\chi_\xi\ast\chi_{\xi'}=\sum_{\xi''}\kappa_{\xi,\xi',\xi''}\chi_{\xi''}.$$
Under the convolution product $\ast$,  $\mathrm{Fun}_G(\mathscr{F}\times\mathscr{F})$ becomes an associative algebra.
A similar construction equips $\mathrm{Fun}_G(\mathscr{F}\times\mathscr{B})$ with a left action of $\mathrm{Fun}_G(\mathscr{F}\times\mathscr{F})$ and a right action of $\mathrm{Fun}_G(\mathscr{B}\times\mathscr{B})$.

Recall $S_q(N,m)$ is the $q$-Schur algebra (of type $A$).
A geometric interpretation of the $q$-Schur algebra and $q$-Schur duality is provided by the following theorem.

 \begin{theorem} [\cite{BLM90}, \cite{GL92}]
 \label{thm:Sa}
 We have the following commutative diagrams:
 \begin{equation*}
\begin{array}{ccccc}
  \mathrm{Fun}_G(\mathscr{F}\times\mathscr{F}) &\curvearrowright
 & \mathrm{Fun}_G(\mathscr{F}\times \mathscr{B}) & \curvearrowleft  \; & \mathrm{Fun}_G(\mathscr{B} \times \mathscr{B})
  \\
  \downarrow\simeq\;\; & & \downarrow\simeq & & \downarrow\simeq
 \\
S_q(N,m)|_{q=-\mathbf{q}^{\frac{1}{2}}}     & \curvearrowright
& \V^{\otimes m}|_{q=-\mathbf{q}^{\frac{1}{2}}} &  \curvearrowleft  \; & \H_q(\mathfrak{S}_m)|_{q=-\mathbf{q}^{\frac{1}{2}}}
\end{array}.
 \end{equation*}
 \end{theorem}

\subsubsection{Type $B$}

Fix a non-degenerate symmetric bilinear form $$(\cdot,\cdot): \mathbb{F}_\mathbf{q}^{2m+1}\times\mathbb{F}_\mathbf{q}^{2m+1}\rightarrow\mathbb{F}_\mathbf{q}.$$
 Let $G=\mathrm{SO}_{2m+1}(\mathbf{q})$ be the orthogonal group associated with $(\cdot,\cdot)$. Let
\[
N=2r+1.
\]
Introduce the full flag varieties of type B:
\begin{align*}
\mathscr{B}^\imath & =\{0=V_0\subset\cdots\subset V_{2m+1}=\mathbb{F}_\mathbf{q}^{2m+1}~|~\dim V_i=i, V_i=V_{2m+1-i}^\perp, \forall i \},
\end{align*}
and the $N$-step isotropic flag variety of type B:
\begin{align*}
\mathscr{F}^\imath & =\{0=V_0\subset\cdots\subset V_{2r+1}=\mathbb{F}_\mathbf{q}^{2m+1}~|~
V_i=V_{2r+1-i}^\perp, 0\leq i\leq 2r+1\},
\end{align*}
where $(-)^\perp$ denotes the orthogonal complement with respect to $(\cdot,\cdot)$.

Similarly, $\mathrm{Fun}_G(\mathscr{F}^\imath\times\mathscr{F}^\imath)$ is an associative algebra under a convolution product. The space $\mathrm{Fun}_G(\mathscr{F}^\imath\times \mathscr{B}^\imath)$ admits a left action of $\mathrm{Fun}_G(\mathscr{F}^\imath\times\mathscr{B}^\imath)$ and a right action of $\mathrm{Fun}_G(\mathscr{B}^\imath\times\mathscr{B}^\imath)$.

Recall the $\imath$Schur algebra $S_q^{\imath}(N,m) =\mathrm{End}_{\H}(\V^{\otimes m})$, where $\dim \V =N=2r+1$, is a homomorphic image of $\Ui_q(\mathfrak{sl}_{2r+1})$ from the $\imath$Schur duality. The following is a geometric interpretation of the $\imath$Schur algebra  and $\imath$Schur duality.

\begin{theorem}  \cite{BKLW18}
 \label{thm:Sb}
We have the following commutative diagrams:
\begin{equation*}
\begin{array}{ccccc}
  \mathrm{Fun}_G(\mathscr{F}^\imath\times\mathscr{F}^\imath) &\curvearrowright
 & \mathrm{Fun}_G(\mathscr{F}^\imath\times \mathscr{B}^\imath) & \curvearrowleft  \; & \mathrm{Fun}_G(\mathscr{B}^\imath\times \mathscr{B}^\imath)
  \\
  \downarrow\simeq\;\; & & \downarrow\simeq & & \downarrow\simeq
 \\
S_q^{\imath}(N,m)|_{q=-\mathbf{q}^{\frac{1}{2}}}     & \curvearrowright
& \V^{\otimes m}|_{q=-\mathbf{q}^{\frac{1}{2}}} &  \curvearrowleft  \; & \H_q(B_m)|_{q=-\mathbf{q}^{\frac{1}{2}}}
\end{array}.
 \end{equation*}
\end{theorem}

\begin{remark}
The $\imath$Schur algebra $S_q^{\imath}(N,m)$, for $N=2r+2$ (which is also a homomorphic image of $\Ui_q(\mathfrak{sl}_{2r+2})$ from the $\imath$Schur duality), also admits a geometric realization using a variant of $N$-step isotropic flag variety of type B.
Moreover, $S_q^{\imath}(N,m)$, for all $N$, can be realized using suitable isotropic flag varieties of type C. See \cite{BKLW18}.
\end{remark}

\subsubsection{Canonical bases revisited}

Thanks to the geometric realization via flag varieties, the canonical (= KL) basis on Hecke algebra can be realized as perverse sheaves supported on Schurbert varieties \cite{KL80}.

Similarly, the canonical bases on the tensor space $\V^{\otimes m}$ and the $q$-Schur algebra $S_q(N,m)$ can be realized as perverse sheaves \cite{BLM90, GL92}. So are the $\imath$-canonical bases on the tensor space $\V^{\otimes m}$ and the $\imath$Schur algebra $S_q^\imath (N,m)$ \cite{BKLW18, FL15}.

Such a perverse sheaf interpretation arising from Schubert varieties of type A (and classical type) also provides an explanation why the ($\imath$-)canonical bases can be used to reformulate the KL theory of type A (and classical type); see Section~\ref{sec:KL}.

Thanks to the geometric interpretation, the ($\imath$-)canonical bases (used in the formulation for KL theory of super classical type) admit various positivity properties.

\subsection{Geometric realizations of modified ($\imath$-)quantum groups}

Beilinson-Lusztig-MacPherson \cite{BLM90} constructed the (modified) quantum group of type A and its canonical basis from the family of $q$-Schur algebras and their canonical bases (via the geometric realization in Theorem~\ref{thm:Sa}); this was carried out through a stabilization procedure which relies on certain explicit multiplication formulas on $q$-Schur algebras with Chevalley generators. This stabilization or inverse limit procedure can be carried out by a different approach via transfer maps \cite{Lus99}. This stabilization can be summarized informally as
\[
\lim_{\stackrel{\longleftarrow}{m}} S_q(N,m)  \cong \dot \U_q(\mathfrak{gl}_N).
\]

The canonical basis on $\U_q(\mathfrak{gl}_N)$ arising this way do not have positivity (see \cite{LiW18}), as a part of the limiting process is not completely geometric. Nevertheless, a variant of this stabilization process (see \cite{Mc12}; also see \cite{LiW18}) leads to
$\dot \U_q(\mathfrak{sl}_N)$ and its canonical basis, which matches completely with the canonical basis defined in \S\ref{subsec:CBQG}. The canonical basis on $\dot \U_q(\mathfrak{sl}_N)$ inherits the positivity property from that on $q$-Schur algebras.

A type B (or C) generalization of the BLM type construction has been carried out in \cite{BKLW18} by studying the multiplicative structures of the family of $\imath$Schur algebras (via the geometric realization in Theorem~\ref{thm:Sb}).
The stabilization procedure in this setting produces the modified $\imath$quantum group and its $\imath$-canonical basis:
\[
\lim_{\stackrel{\longleftarrow}{m}} S_q^\imath (N,m)  \cong \dot \Ui_q(\mathfrak{gl}_N).
\]
(The case for $N$ even is more subtle; see \cite[Appendix A]{BKLW18}.) The $\imath$-canonical basis on $\dot \Ui_q(\mathfrak{gl}_N)$ do not admit the positivity. A variant of the stabilization procedure produces $\Ui_q(\mathfrak{sl}_N)$ with an $\imath$-canonical basis which exhibits the positivity.

%

\begin{remark}
The results in this section admit affine generalizations.
We refer to \S\ref{subsec:out} for discussions and references.
\end{remark}



\end{document}